\documentclass[12pt,reqno]{amsart}
\usepackage[T1]{fontenc}
\usepackage{amsfonts}
\usepackage{amssymb}
\usepackage{mathrsfs}
\usepackage[latin1]{inputenc}
\usepackage{amsmath}
\usepackage[english]{babel}
\usepackage{setspace}
\usepackage[usenames, dvipsnames]{color}
\usepackage{soul}

 \newtheorem{thm}{Theorem}[section]
 \newtheorem{thmalt}{Theorem}[thm]
 \newenvironment{thmprime}[1]{\renewcommand\thethmalt{#1}\thmalt}{\endthmalt}
 \newtheorem{cor}[thm]{Corollary}
 
 \newtheorem{prop}[thm]{Proposition}
 \theoremstyle{definition}
 \newtheorem{defn}[thm]{Definition}
 \newtheorem{rem}[thm]{Remark}
 \newtheorem*{ex}{Example}
 \numberwithin{equation}{section}

 \numberwithin{equation}{section}

\newcommand{\R}{{\mathbb R}}

\newcommand{\N}{{\mathbb N}}

\newcommand{\cF}{{\mathcal F}}

\newcommand{\cS}{{\mathcal S}}
\newcommand{\cO}{{\mathcal O}}

\newcommand{\su}{\subseteq}

\def\Wig{{\mathop{\rm Wig }}}
\def\BJ{{\mathop{\rm BJ }}}

\begin{document}

\title[Mutual estimate of representations and uncertainty principles]{Mutual estimates of time-frequency representations and uncertainty principles}

\author{Angela A. Albanese, Claudio Mele and Alessandro Oliaro}

\address{Angela A. Albanese, Dipartimento di Matematica e Fisica ``E. De Giorgi'', Universit\`a del Salento - C.P.193, I-73100 Lecce, Italy}
\address{Claudio Mele, Dipartimento di Matematica e Fisica ``E. De Giorgi'', Universit\`a del Salento - C.P.193, I-73100 Lecce, Italy}
\address{Alessandro Oliaro, Dipartimento di Matematica ``G. Peano'', Universit\`a di Torino - Via Carlo Alberto, 10, I-10123 Torino, Italy}

\begin{abstract}
In this paper we give different estimates between Lebesgue norms of quadratic time-frequency representations. We show that, in some cases, it is not possible to have such bounds in classical $L^p$ spaces, but the Lebesgue norm needs to be suitably weighted. This leads to consider weights of polynomial type, and, more generally, of ultradifferentiable type, and this, in turn, gives rise to use as functional setting the ultradifferentiable classes. As applications of such estimates we deduce uncertainty principles both of Donoho-Stark type and of local type for representations. \\[0.1cm]
{\bf Keywords:} Quadratic forms, time-frequency representations, uncertainty principles, ultradifferentiable spaces. \\[0.1cm]
{\bf 2020 Mathematics Subject Classification:} 42B10, 46F05, 46E30, 46F12.
\end{abstract}

\maketitle

\section{Introduction }\label{intro}

The classical Fourier analysis on $\mathbb{R}^N$ is based on the Fourier transform, defined as
$$
\hat{f}(\xi):=\int_{\mathbb{R}^N} e^{-ix\xi} f(x)\,dx,
$$
for $f\in L^1(\mathbb{R}^N)$, with standard extension to more general spaces of functions and distributions. In signal analysis the Fourier transform is interpreted as an instrument showing the frequencies that a certain signal $f$ contains. A drawback is that the information on when different frequencies occur in a (non-stationary) signal $f$ is hidden in the complex phase of $\hat{f}$, and this has led, in the last decades, to the proposal of several different time-frequency representations, in order to make this information more readable and clear. A time-frequency representation of a signal $f$ is a function (or distribution) of $(x,\xi)\in\mathbb{R}^{2N}$, giving the energy of the signal at time $x$ and frequency $\xi$; it is usually a quadratic form of $f$, and it is not uniquely chosen. There are in fact many different representations, each one with good features and disadvantages, and different representations better address different problems. This of course has to do with the use of such tools in applications, but also with their more abstract mathematical properties, as well as with the role that such representations have in other fields of Mathematics, such as partial differential equations, group representations, operator theory and so on. Since all such representations have common aims and meanings, it is reasonable to investigate in which aspects they are linked and how they can be, in some sense, \emph{controlled} one each other. Of course there are several well-known connections between different representations; there are classes of time-frequency distributions underlying common structures, and many comparisons have been made in the literature, showing similarities and differences.

In this paper we consider a new perspective, that has to do with the possibility to estimate the (weighted) $L^p$-norm of a time-frequency representation by a (weighted) $L^q$-norm of another one. This is motivated by the observation that the Donoho-Stark uncertainty principle for the short-time Fourier transform can be rephrased in terms of estimates between Lebesgue norms of the spectrogram and the Rihaczek form. We consider then the general problem of giving mutual estimates of couples of time-frequency representations. We prove that some of these estimates are not true when considering usual $L^p$ spaces, but are fulfilled when considering weighted $L^p$ spaces, for weights of polynomial type or more generally of exponential (ultradifferentiable) type. As an application, we prove new uncertainty principles of Donoho-Stark type involving time-frequency representations, as well as uncertainty principles of local type for quadratic forms.

In order to better explain motivations and results of this paper we recall some basic facts.

The classical Donoho-Stark uncertainty principle is based on the definition of $\varepsilon$ concentration; we say that a function $g\in L^2(\mathbb{R}^N)$ is $\varepsilon$-concentrated on a measurable set $V\subset\mathbb{R}^N$, $\varepsilon\in[0,1]$, if
$$
\int_V |f(x)|^2\,dx\geq (1-\varepsilon^2)\| f\|_2^2.
$$
\begin{thm}[Donoho-Stark uncertainty principle \cite{DS}]
Let $T,\Omega\subset\mathbb{R}^N$ be measurable sets, and $\varepsilon_T,\varepsilon_\Omega\geq 0$ with $\varepsilon_T+\varepsilon_\Omega\leq 1$. If there exists $f\in L^2(\mathbb{R}^N)$, $f\neq 0$, such that $f$ is $\varepsilon_T$-concentrated on $T$ and $\hat{f}$ is $\varepsilon_\Omega$-concentrated on $\Omega$, then
$$
m(T)m(\Omega)\geq (2\pi)^N (1-\varepsilon_T-\varepsilon_\Omega),
$$
where $m(\cdot)$ indicates the Lebesgue measure of the corresponding set.
\end{thm}
There is a corresponding version of the Donoho-Stark principle for the \textit{short-time Fourier transform}. Recall that the short-time Fourier transform is defined as
\begin{equation}\label{stftintro}
V_g f(x,\xi):=\int_{\mathbb{R}^N} f(y) \overline{g(y-x)} e^{-iy\xi}\,dy,
\end{equation}
for $f\in\mathcal{S}(\mathbb{R}^N)$, with standard extensions to more general spaces of functions and distributions. We refer to \cite[Chapter 3]{G} for a complete treatment of the basic results on this transformation. Here we just recall that from H\"older inequality, if $f,g\in L^2(\mathbb{R}^N)$ we immediately have that $V_g f\in L^\infty(\mathbb{R}^{2N})$ and
\begin{equation}\label{estVgf2}
	\| V_gf\|_\infty\leq \| f\|_2 \| g\|_2,
\end{equation}
where the norm in the left-hand side is in $\mathbb{R}^{2N}$ and the ones in the right-hand side are in $\mathbb{R}^N$. An uncertainty principle of Donoho-Stark type can be proved for the short-time Fourier transform, in the following form (\cite[Proposition 3.3.1]{G}).
\begin{thm}\label{DSVgfintro}
	Let $E\subset\mathbb{R}^{2N}$ be a measurable set and $\varepsilon>0$. Suppose that there exist $f,g\in L^2(\mathbb{R}^N)$, $f,g\neq 0$, such that
	\begin{equation}\label{thDSVg}
	\int_E |V_g f(x,\xi)|^2\,dx\,d\xi\geq (1-\varepsilon) \| f\|_2^2 \| g\|_2^2.
	\end{equation}
	Then $m(E)\geq 1-\varepsilon$.
\end{thm}
An inspection of the proof of Theorem \ref{DSVgfintro} reveals that such a result is strictly related to \eqref{estVgf2}. Both \eqref{estVgf2} and \eqref{thDSVg} can be easily rewritten in terms of time-frequency representations. In Section \ref{section-definitions} we give more information, and we analyze the main representations used in the literature. Here we just recall that the most natural representation is the so-called \textit{Rihaczek} form, defined as
$$
Rf(x,\xi):=e^{-ix\xi}f(x)\overline{\hat{f}(\xi)};
$$
moreover, a widely used time-frequency representation, related to the short-time Fourier transform, is the \textit{spectrogram}, defined as
\begin{equation}\label{defSpectrogram}
Sp_g f(x,\xi):=| V_g f(x,\xi)|^2.
\end{equation}
Now, since $\| Sp_gf\|_\infty=\|V_gf\|_\infty^2$ and $\|Rf\|_2=\| f\|_2 \|\hat{f}\|_2=(2\pi)^{N/2}\| f\|_2^2$, we have that \eqref{estVgf2} can be rewritten as
\begin{equation}\label{estrepr}
\| Sp_g f\|_\infty\leq (2\pi)^{-N}\| Rf\|_2 \| Rg\|_2,
\end{equation}
where in this last estimate all the norms are in $\mathbb{R}^{2N}$. Furthermore, the Donoho-Stark type uncertainty principle of Theorem \ref{DSVgfintro} can be rephrased as follows.
\begin{thmprime}{\ref{DSVgfintro}$'$}
Let $E\subset\mathbb{R}^{2N}$ be a measurable set and $\varepsilon>0$. Suppose that there exist $f,g\in L^2(\mathbb{R}^N)$, $f,g\neq 0$, such that
\begin{equation}\label{hpDSmodif}
\int_E |Sp_g f(x,\xi)|\,dx\,d\xi\geq (1-\varepsilon) (2\pi)^{-N} \| Rf\|_2 \| Rg\|_2.
\end{equation}
Then $m(E)\geq 1-\varepsilon$.
\end{thmprime}
\noindent Taking a window $g$ such that $\| g\|_2=1$, this result, roughly speaking, says that if the ($L^1$) contents of the spectrogram of some signal $f$ in a set $E$ is larger than a fraction of the whole ($L^2$) contents of the Rihaczek of $f$, then the measure of $E$ must be sufficiently large. A natural question is if we can replace in \eqref{hpDSmodif} the spectrogram and the Rihaczek with other time-frequency representations. This, in turn, depends on which kind of estimates of the type \eqref{estrepr} we can prove on different couples of representations. \\[0.2cm]
\indent In this paper we then investigate estimates between different representations in the frame of (weighted) $L^p$-spaces, and deduce different kind of uncertainty principles. It is usually not difficult to estimate a representation by the Rihaczek, since (as in the case of the spectrogram) this can often be obtained as a consequence of the mapping properties in Lebesgue spaces of the representation itself. On the other hand, when in the right-hand side we want to put something different from the Rihaczek, the things become much more difficult, and some estimates are not true anymore for $L^p$ norms. For instance, we prove that for $p\geq 2$ there does not exist a constant $D>0$ such that
\begin{equation}\label{stimavietata}
	\|Rf\|_\infty\leq D\|\Wig (f)\|_p
\end{equation}
for every $f\in\mathcal{S}(\mathbb{R}^N)$, where
\begin{equation}\label{defWigner}
	\Wig (f)(x,\xi):=\int_{\R^N}e^{- i t\xi} f\left(x+\frac{t}{2}\right)\overline{f\left(x-\frac{t}{2}\right)}\, dt
\end{equation}
is the classical \textit{Wigner transform}. On the other hand, \eqref{stimavietata} becomes true if we consider weighted Lebesgue norms instead of unweighted ones; for instance, fixed $1\leq p\leq\infty$, there exist $C,\mu >0$ such that for every $f\in\mathcal{S}(\mathbb{R}^N)$ we have
$$
\|Rf\|_\infty\leq C\|(1+|x|^2+|\xi|^2)^\mu\, \Wig (f)\|_p.
$$
More generally, it turns out that ultradifferentiable weights allow us to obtain good estimates between Lebesgue norms of different representations, so a natural functional setting is given by ultradifferentiable spaces. We then consider weight functions in the sense of Braun-Meise-Taylor \cite{BMT} and corresponding (global) spaces of rapidly decreasing ultradifferentiable functions as defined in \cite{B}; weighted $L^p$-norms become, in fact, seminorms in the corresponding ultradifferentiable class (see \cite{BJOrpwt}), and we prove mutual estimates on weighted $L^p$-norms of different time-frequency representations.

Besides the interest that such estimates have in themselves, they have different applications. We prove indeed that they can be used to obtain uncertainty principles of Donoho-Stark type for different time-frequency representations, where the hypothesis \eqref{hpDSmodif} is substituted by the request that the contents of a time-frequency representation in some Lebesgue measurable set $E$ is greater than a fraction of the whole contents of another representation.

Moreover, mutual estimates between representations constitute the basic tool to prove uncertainty principles of local type, in the spirit of \cite{P}, \cite{BCO}, for different time-frequency distributions. We recall that, for the Fourier transform, the local uncertainty principle as formulated in \cite{P} is an inequality that, in its simplest version in dimension $1$, reads as
\begin{equation}\label{intro-loc-up}
\int_E |\hat{f}(\xi)|^2\,d\xi < K' m(E) \Delta(f)
\end{equation}
for every measurable set $E\subset\mathbb{R}$ and $f\in L^2(\mathbb{R})$ with $\| f\|_2=1$, where $K'$ does not depend on $E$ and $f$, $m(E)$ is the Lebesgue measure of $E$, and
$$
\Delta(f)=\| (t-\mu(f)) f(t)\|_2
$$
is the dispersion associated with $f$, for $\mu(f)=\int_{\mathbb{R}} t|f(t)|^2\,dt$. Then \eqref{intro-loc-up} says that $\hat{f}$ cannot show, locally in a measurable set $E$, a large amount of energy if $E$ has small measure and/or $f$ is very concentrated. Results of this type are proved in \cite{BCO} in the frame of time-frequency analysis. In this paper we prove uncertainty principles of local type for time-frequency representations, giving limitations of the amount of energy that a time-frequency representation of $f$ may have on a measurable set $E$ in terms of quantities measuring the size of $E$ and the dispersions of $f$ and/or $\hat{f}$; we also give analogous results where the role of $f$, $\hat{f}$, and the time-frequency representation of $f$ are interchanged. Moreover, the fact that all estimates are proved in weighted spaces implies that weights appear in the quantities measuring the size of $E$, too; in some cases this \emph{weighted measure} of $E$ may be finite even when the Lebesgue measure of $E$ is infinite, enlarging then the class of sets $E$ for which our results are meaningful.\\[0.2cm]
\indent The paper is organized as follows. In Section \ref{section-definitions} we recall definitions and basic properties of time-frequency representations and of ultradifferentiable spaces. In Section \ref{section-estimates} we prove different estimates between weighted norms of representations, and the corresponding uncertainty principles of Donoho-Stark type. Finally, Section \ref{Sec-local-up} is devoted to the study of local uncertainty principles.

\section{Definitions and preliminary results}\label{section-definitions}


\begin{defn}\label{D.weight} A weight function is a continuous increasing function $\omega:[0,\infty)\to[0,\infty)$  satisfying the following properties:
	\begin{itemize}
		\item[($\alpha$)] there exists $K\geq1$ such that $\omega(2t)\leq K(1+\omega(t))$ for every $t\geq0$;
		\item[($\beta$)] $\omega(t)=o(t)$ as $t \to \infty$;
		\item[($\gamma$)] there exist $a \in \mathbb{R}$, $b>0$ such that $\omega(t) \geq a+b\log(1+t)$ for every $t\geq0$;
		\item[($\delta$)] $\varphi_\omega(t):= \omega (e^t)$ is a convex function.
	\end{itemize}
\end{defn}
Given a weight function we can extend $\omega: \mathbb{R} \to [0,\infty)$ by defining $\omega(x)=\omega(|x|)$  for all $x \in \mathbb{R}$ (of course, in the same way we could extend $\omega $ to $\mathbb{R}^N$ for every $N$). The condition $(\beta)$ is weaker than the condition of non-quasianalyticity $\int_{1}^{\infty} \frac{\omega(t)}{1+t^2}\, dt < \infty$. When the latter condition is satisfied, the spaces that we are going to define shall contain non trivial compactly supported functions. Most of the results of this paper hold under condition $(\beta)$, i.e., both in the non-quasianalytic and in the quasianalytic case. When necessary, we will specify whether the weight has to be considered non-quasianalytic. 

Sometimes, we assume that the weight $\omega$ satisfies the additional condition $\log (1+t)=o(\omega(t))$ as $t\to\infty$, called condition $(\gamma')$, which is stronger than condition $(\gamma)$. 
\\[0.1cm]

	We recall some known properties of the weight functions that shall be useful in the following (the proofs can be found in the literature): 
	

	(1) Condition $(\alpha)$ implies for every $t_1, t_2 \geq 0$ that
	\begin{equation}\label{sub}
	\omega(t_1+t_2)\leq K(1+\omega(t_1)+\omega(t_2)).
	\end{equation}
	Observe that this condition is weaker than subadditivity (i.e.,  $\omega(t_1+t_2)\leq \omega(t_1) + \omega(t_2))$. The weight functions satisfying ($\alpha$) are not necessarily subadditive in general.
	
	(2) Condition $(\alpha)$ implies that there exists $L\geq 1$ such that for every $t\geq 0$
	\begin{equation*}
	\omega(e t)\leq L(1+\omega(t)).
	\end{equation*}
	
	(3) By condition $(\gamma)$ we have  for every  $\lambda>\frac{N}{bp}$ that
	\begin{equation}\label{eq.Lpspazi}
	e^{-\lambda \omega(t)}\in L^p(\R^N), \quad 1\leq p<\infty.
	\end{equation}

Given a weight function $\omega$, we define the Young conjugate $\varphi^*_\omega$ of $\varphi_\omega$ as the function $\varphi^*_\omega:[0,\infty)\to [0,\infty)$,
\begin{equation*}\label{Yconj}
\varphi^*_\omega(s):=\sup_{t\geq 0}\{ st-\varphi_\omega(t)\},\quad s\geq 0.
\end{equation*}
There is no loss of generality to assume that $\omega$ vanishes on $[0,1]$. We have that $\varphi^*_\omega$ is convex and increasing, $\varphi^*_\omega(0)=0$ and $(\varphi^*_\omega)^*=\varphi_\omega$. 

Now, we introduce the ultradifferentiable function space $\cS_\omega(\R^N)$ in the sense of Bj\"ork \cite{B}.
\begin{defn}\label{D.Beurling} Let $\omega$ be a weight function.
	We denote by  $\mathcal{S}_\omega(\mathbb{R}^N)$  the set of all functions $f\in L^1(\mathbb{R}^N)$ such that $f,\hat{f}\in C^\infty(\mathbb{R}^N)$ and for each $\lambda >0$ and $\alpha\in\N_0^N$ we have
	\begin{equation}\label{cond Sw 1}
	\| e^{\lambda\omega}\partial^\alpha f\|_\infty <\infty\ \ {\rm and }\ \  
	\|e^{\lambda\omega}\partial^\alpha\hat{f}\|_\infty <\infty \; ,
	\end{equation}
	where $\hat{f}$ denotes the Fourier transform of $f$. The elements of  $\mathcal{S}_\omega(\mathbb{R}^N)$ are called \textit{$\omega$-ultradifferentiable rapidly decreasing functions of Beurling type}. 
	We denote by  $\cS'_{\omega}(\R^N)$ the dual of $\cS_{\omega}(\R^N)$  endowed with its strong topology.
\end{defn}
Now, we recall some properties of  $\cS_\omega(\R^N)$.

\begin{rem}\label{in s} Let $\omega$ be a weight function. 
	
	(1)	The condition $(\gamma)$ of Definition \ref{D.weight} implies that  $\mathcal{S}_\omega(\mathbb{R}^N)\su \mathcal{S}(\mathbb{R}^N)$  with continuous inclusion. The spaces coincide when $\omega(t)=\log(1+t)$, for $t\geq 0$.
	Accordingly,  we can rewrite the definition of $\mathcal{S}_\omega(\mathbb{R}^N)$ as the set of all the Schwartz functions that satisfy the condition  $(\ref{cond Sw 1})$. 
	
	(2)	The space $\mathcal{S}_\omega(\mathbb{R}^N)$ is closed under convolution and under point-wise multiplication.
	
	(3) Let us denote by $T_x, M_\xi$ and $\Pi(z)$, respectively, the \textit{translation}, the \textit{modulation} and the \textit{phase-space shift} operators, defined by
	\begin{align*}
	&T_x f(y):= f(y-x),\\& M_\xi f(y):= e^{iy\xi}f(y), \\& \Pi(z)f(y):=M_\xi T_x f(y)= e^{iy \xi} f(y-x)
	\end{align*}
	for $x,y,\xi \in \R^N$ and $z=(x,\xi)$. The space $\mathcal{S}_\omega (\mathbb{R}^N)$ is closed under translation and modulation, and hence under the action of the phase-space shift operator.

	(4)	The Fourier transform $\mathcal{F}:\mathcal{S}_{\omega}(\mathbb{R}^N)\to \mathcal{S}_{\omega}(\mathbb{R}^N)$ is a continuous isomorphism.

\end{rem}

The space $\cS_\omega(\R^N)$ is a Fr\'echet space, and we can give different equivalent systems of seminorms. We refer to \cite{ AC2,BJOrpwt} for more results in this direction; here we just recall the following proposition.

\begin{prop}\label{P.norme}
	Let $\omega$ be a weight function and consider $f \in \mathcal{S}(\mathbb{R}^N)$. Fix $1\leq p\leq \infty$. Then $f\in\cS_\omega(\R^N)$ if, and only if, one of the following  conditions is satisfied.
	\begin{enumerate}	
		\item	\begin{enumerate}	
			\item[{\rm (i)} ] $\forall\lambda>0, \; \alpha\in\mathbb{N}^N_0$ $\exists C_{\alpha,\lambda,p}>0$ such that 
			$\|e^{\lambda\omega}\partial^\alpha f\|_p\leq C_{\alpha,\lambda,p}$, and
			
			\item[{\rm (ii)}] $\forall\lambda>0, \; \alpha\in\mathbb{N}^N_0$  $\exists C_{\alpha,\lambda,p}>0$ such that 	
			$\| e^{\lambda\omega}\partial^\alpha \hat{f}\|_p\leq C_{\alpha,\lambda,p} $.
		\end{enumerate}
		\item \begin{enumerate}
			\item[\rm (i)] $\forall \lambda>0$ $\exists C_{\lambda,p}>0$ such that
			$\| e^{\lambda\omega} f\|_p \leq C_{\lambda,p}$, and
			\item[\rm (ii)]	$\forall \lambda>0$ $\exists C_{\lambda,p}>0$ such that
			$\|e^{\lambda\omega}\hat{f}\|_p\leq C_{\lambda,p} $.
		\end{enumerate}
	\end{enumerate}
\end{prop}
Now, we recall the definition of the spaces $\cO_{M,\omega}(\R^N)$ and $\cO_{C,\omega}(\R^N)$, that have been introduced in \cite{AC,AC2}.

\begin{defn}Let $\omega$ be a non-quasianalytic weight function.
	
	(a) 	We denote by $\mathcal{O}_{M,\omega}(\mathbb{R}^N)$  the set of all functions $f\in {C}^\infty(\mathbb{R}^N)$ such that for all $m\in \mathbb{N}$ there exist $C>0$ and $n \in\mathbb{N}$ such that for every $\alpha \in \mathbb{N}^N_0$ and $x\in \mathbb{R}^N$ we have
	\begin{equation*}\label{m}
	|\partial^\alpha f(x)|\leq Ce^ {n\omega(x)}e^{m\varphi^*_\omega\left(\frac{|\alpha|}{m}\right)}.		
	\end{equation*}
	The elements of $\mathcal{O}_{M,\omega}(\mathbb{R}^N)$ are called  \textit{ slowly increasing functions of Beurling type}.

	(b) We denote by  $\cO_{C,\omega}(\mathbb{R}^N)$  the set of all functions $f\in C^\infty(\R^N)$ for which there exists $n\in \mathbb{N}$ such that for all $m \in\mathbb{N}$ there exists $C>0$ so that for every $\alpha \in \mathbb{N}^N_0$ and $x\in \mathbb{R}^N$ we have
	\begin{equation*}\label{c}
	|\partial^\alpha f(x)|\leq Ce^{n\omega(x)}e^{m\varphi^*_\omega\left(\frac{|\alpha|}{m}\right)}.		
	\end{equation*}
	The elements of $\mathcal{O}_{C,\omega}(\mathbb{R}^N)$ are called  \textit{very slowly increasing functions of Beurling type}.
\end{defn}

The space $\mathcal{O}_{M,\omega}(\mathbb{R}^N)$ is the space of multipliers of $\cS_\omega(\R^N)$
 and of its dual space $\cS'_\omega(\R^N)$ 
as proved in \cite{AC}, i.e., $f\in \cO_{M,\omega}(\R^N)$, if and only if $fg\in \cS_\omega(\R^N)$ for all $g\in \cS_\omega(\R^N)$
 if and only if $fT\in \cS'_{\omega}(\R^N)$ for all $T\in \cS'_\omega(\R^N)$.

The space $\mathcal{O}'_{C,\omega}(\mathbb{R}^N)$ is the space of convolutors of $\cS_\omega(\R^N)$, as the following result shows.
\begin{prop}[{\cite[Theorem 5.3]{AC2}}]\label{T.conv}
	Let $\omega$ be a non-quasianalytic weight function and $T\in \cS'_\omega(\R^N)$. Consider the following properties:
	\begin{enumerate} 
		\item $T\in \cO'_{C,\omega}(\mathbb{R}^N)$.
		\item For every $f\in \mathcal{S}_\omega(\mathbb{R}^N)$, we have $T\star f \in \mathcal{S}_\omega(\mathbb{R}^N)$.  \end{enumerate}
	Then $(1)\Rightarrow (2)$. If, in addition, the weight function $\omega$ satisfies the stronger condition $(\gamma')$, then $(2)\Rightarrow (1)$.
\end{prop}

Given a non-quasianalytic weight function $\omega$, the  Fourier transform $\mathcal{F}: \mathcal{O}'_{C,\omega}(\mathbb{R}^N)\to  \cO_{M,\omega}(\R^N)$ is well-defined. Moreover, for every $T\in \cO'_{C,\omega}(\R^N)$ and $f\in \cS_\omega(\R^N)$,  the convolution $T\star f$ satisfies the following property:
\begin{equation}\label{eq.FT-C}
\cF(T\star f)=\hat{f}\cF(T).
\end{equation}
If, in addition, the weight $\omega$ satisfies the stronger condition $(\gamma')$, then $\cF$ is a topological isomorphism from  $\cO'_{C,\omega}(\R^N)$  onto $\cO_{M,\omega}(\R^N)$, as shown in  \cite[Theorem 6.1]{AC2}. 

The short-time Fourier transform \eqref{stftintro} can be easily extended to ultradistributions; indeed, let us consider a window function $g\in \mathcal{S}_{\omega}(\mathbb{R}^N)$, $g\neq0$. The short-time Fourier transform (briefly STFT) of $f\in \mathcal{S}_\omega'(\mathbb{R}^N)$ is defined, for $z=(x,\xi)\in \R^{2N}$, by:
\begin{equation}\label{vgf}
V_g f(z):= \langle f, \Pi(z) g \rangle= \int_{\R^N} f(y) \overline{g(y-x)}e^{-i y \xi} \,dy,
\end{equation}
where the bracket $\langle \cdot,\cdot \rangle$ and the integral in \eqref{vgf} denote the conjugate linear action of $\mathcal{S}_\omega'(\mathbb{R}^N)$ on $\mathcal{S}_\omega(\mathbb{R}^N)$, consistent with the inner product $( \cdot,\cdot )_2$ on $L^2(\R^N)$. The definition is well posed, because if $g\in \mathcal{S}_\omega(\mathbb{R}^N)$, then $\Pi(z)g\in \mathcal{S}_\omega(\mathbb{R}^N)$ for each $z=(x,\xi) \in \mathbb{R}^{2N}$.  
For basic properties of the STFT in ultradifferentiable spaces see \cite{GZ}. We recall the \textit{inversion formula}:
\begin{align}\label{inv}
f(y)=\frac{1}{(2\pi)^N \|g\|_2^2}\int_{\R^{2N}} V_g f(z) (\Pi(z)g)(y)\, dz.
\end{align}
The STFT is also well-defined and continuous on $\cS_\omega(\R^N)$, as the following results contained in \cite{GZ} shows.
\begin{thm}\label{equiv vgf}
	Let $\omega$ be a weight function and consider  $g\in \mathcal{S}_{\omega}(\mathbb{R}^N)$, $g\neq0$. Then for $f\in \mathcal{S}_{\omega}'(\mathbb{R}^N)$ the following conditions are equivalent:
	\begin{itemize}
		\item[(1)] $f\in \mathcal{S}_{\omega}(\mathbb{R}^N)$.
		\item[(2)] For each $\lambda>0$ there exists a constant $C_\lambda>0$ such that
		\begin{equation*}\label{decay Vgf}
		|V_gf(z)|\leq C_\lambda e^{-\lambda\omega(z)}
		\end{equation*}
		for each $z\in \mathbb{R}^{2N}$.
		\item[(3)] $V_g f\in\cS_\omega(\mathbb{R}^{2N})$.
	\end{itemize}
\end{thm}
 Therefore, we can deduce the following result.
 \begin{prop}\label{cont Vg}
 	Let $\omega$ be a weight function and consider  $g\in \mathcal{S}_{\omega}(\mathbb{R}^N)$, $g\neq0$. Then
 	\begin{equation*}
 	V_g: \mathcal{S}_{\omega}(\mathbb{R}^N)\to \mathcal{S}_{\omega}(\mathbb{R}^{2N})
 	\end{equation*}
 	is continuous.
 \end{prop}
The STFT also provides new equivalent systems of seminorms for $\mathcal{S}_{\omega}(\mathbb{R}^N)$ (see \cite{BJOrpwt}).
\begin{prop}\label{seminorma vgf}
	Let $\omega$ be a weight function and consider  $g\in \mathcal{S}_{\omega}(\mathbb{R}^N)$, $g\neq0$. Then $f\in\mathcal{S}_\omega(\mathbb{R}^N)$ if and only if one of the following conditions is satisfied:
	\begin{itemize}
		\item[(1)] For all $\lambda\geq 0$
	\begin{equation*}
	\|e^{\lambda\omega}V_g f\|_{\infty}= \sup_{z\in \R^{2N}} |V_g f(z)|e^{\lambda\omega(z)}<\infty.
	\end{equation*}
		\item[(2)] There exists $1\leq p<\infty$ such that for all $\lambda\geq 0$
	\begin{equation*}
	\|e^{\lambda\omega}V_g f\|^p_{p}= \int_{\R^{2N}} |V_g f(z)|^pe^{p\lambda\omega(z)}\,dz<\infty.
	\end{equation*}
	\end{itemize}
\end{prop}

Now, we recall the definition of further time-frequency representations that we shall treat in the following, and that are widely used in the literature, see, for instance, \cite{C,G} and the references therein.
\begin{defn}\label{deftauW}
	For $\tau\in[0,1]$ and $f\in\cS(\R^N)$, we define the \textit{$\tau$-Wigner transform} as
	\begin{align*}
	\Wig_\tau (f)(x,\xi):=\int_{\R^N}e^{- i t\xi} f(x+\tau t)\overline{f(x-(1-\tau)t)}\, dt
	\end{align*}
	for $(x,\xi)\in\R^{2N}$. 
\end{defn}
If $\tau=\frac{1}{2}$, we get the classical Wigner transform \eqref{defWigner}.
In the cases $\tau=0$ and $\tau=1$, we get the Rihaczek and conjugate Rihaczek forms, given by
\begin{align*}
&Rf(x,\xi):=e^{- i x\xi} f(x)\overline{\hat{f}(\xi)},\\& R^*f(x,\xi):=e^{i x\xi}\overline{f(x)}\hat{f}(\xi),
\end{align*}
respectively, for $(x,\xi)\in\R^{2N}$. We have already defined the spectrogram \eqref{defSpectrogram}; it is known that 
the spectrogram and the $\tau$-Wigner transform are linked. In fact, given $f,g\in\cS_\omega(\R^N)$, $g\neq0$, we have
\begin{align}\label{spwig}
Sp_g f=\Wig_{1-\tau}(\tilde{g})\star \Wig_\tau (f),
\end{align}
where $\tilde{g}(t):=g(-t)$, $t\in\mathbb{R}^N$. For more details we refer the reader to \cite{BODD,BODD2,G}.

The $\tau$-Wigner transform and the spectrogram are particular cases of a large class of covariant quadratic representations, the so-called \textit{Cohen class}.
\begin{defn}\label{NV-DefQ}
	Let $\omega$ be a weight function.	Given a \emph{kernel} $\sigma\in \mathcal{S}'_\omega(\mathbb{R}^{2N})$, we define
	$$
	Q_\sigma f:=\sigma\star \Wig (f),
	$$
	for each $f\in \mathcal{S}_\omega(\mathbb{R}^N)$. We say that $Q_\sigma$ is a time-frequency representation in the Cohen class with kernel $\sigma$. 
\end{defn}
Of course, by \cite[Proposition 5.2]{AC2} every time-frequency representation in the Cohen class with kernel in $\mathcal{S}'_\omega(\mathbb{R}^{2N})$ is well-defined for $f\in\mathcal{S}_\omega(\mathbb{R}^{N})$. Moreover, if we consider the weight $\omega(t)=\log(1+t)$, we recover the classic time-frequency representation in the Cohen class with kernel in $\mathcal{S}'(\mathbb{R}^{2N})$.

When the kernel $\sigma$ is given by
$$
\sigma_{\BJ}(x,\xi):=\mathcal{F}_{\substack{t\to x \\ \eta\to\xi}}\left(\frac{2\sin(t\eta/2)}{t\eta}\right)
$$
the corresponding Cohen class form is the so-called \textit{Born-Jordan} representation
\begin{equation}\label{BJdef}
\BJ(f):=\sigma_{\BJ}\star \Wig(f),
\end{equation}
that has many interesting properties (see, for instance, \cite{CR,DG,T} and the references therein). The Born-Jordan representation is linked to the $\tau$-Wigner; it can, indeed, be written as
\begin{equation}\label{BJexpression}
\BJ(f)=\int_0^1 \Wig_\tau(f)\,d\tau,
\end{equation}
for every $f\in\mathcal{S}(\mathbb{R}^N)$, see, for instance, \cite{BODD}.

\section{Estimates on time-frequency representations and Donoho-Stark uncertainty inequalities}\label{section-estimates}
In this section we give some estimates between different time-frequency representations in the Cohen class (STFT, spectrogram, Rihaczek form, $\tau$-Wigner transform); moreover, we prove that, given two time-frequency representations $Q_{\sigma_1}$ and $Q_{\sigma_2}$ in the Cohen class, under suitable conditions on the kernels $\sigma_1$ and $\sigma_2$, it is possible to estimate one representation with the other. Such estimates are in the form of a control of the (weighted) $L^\infty$-norm of a representation by (weighted) $L^p$-norms of another representation; this gives as a consequence new inequalities of Donoho-Stark type.
\vskip0.1cm
Throughout this section $K$ and $b$ always indicate the constants of conditions $(\alpha)$ and $(\gamma)$ of Definition \ref{D.weight}, respectively. Moreover, we use as standard the notation $p'$ to indicate the conjugate exponent of $p\in[1,\infty]$, in the sense that $\frac{1}{p}+\frac{1}{p'}=1$; with this notation, $(2p)'$ is the number such that $\frac{1}{2p}+\frac{1}{(2p)'}=1$, and for $p=\infty$ we intend $(2p)'=1$. The norms in $L^p$ are indicated as $\|\cdot\|_p$. Sometimes they are norms in $L^p(\R^N)$ and sometimes in $L^p(\R^{2N})$, and it is clear from the context which is the case; when this is not clear, we specify by writing $\|\cdot\|_{L^p(\R^N)}$ or $\|\cdot\|_{L^p(\R^{2N})}$.

The next result gives a mutual control between the Spectrogram and the Rihaczek form.

\begin{thm}\label{SptstimaconR}
	Let $\omega$ be a weight function, $g\in \mathcal{S}_{\omega}(\mathbb{R}^N)$, $g\neq0$, and $1\leq p\leq \infty$.
	\begin{itemize}
	\item[(i)]
	For every $\lambda\geq0$ and $f\in\cS_\omega(\R^N)$ we have
	\begin{equation}\label{Blambda}
		\|e^{\lambda\omega}Sp_gf\|_\infty\leq D_\lambda^{(1)}\|e^{2K^2\lambda\omega}Rf\|_{p}\|e^{2K^2\lambda\omega}Rg\|_{p'},  
	\end{equation}
	where $D_\lambda^{(1)}=(2\pi)^{-N} e^{K(1+2K)\lambda}$.
	\item[(ii)]
	For every $\lambda\geq0$, $\mu>\frac{2N}{b(2p)'}$ and $f\in\cS_\omega(\R^N)$ we have
	\begin{equation}\label{Clambda}
	\|e^{\lambda\omega}Rf\|_{\infty}\leq D_{\lambda,\mu}^{(2)}\|e^{2(K^2\lambda+\mu)\omega}Sp_gf\|_{p},
	\end{equation}
	where
	$$
	D_{\lambda,\mu}^{(2)}=\frac{e^{K(1+2K)\lambda} \| e^{K^2\lambda\omega}g\|_\infty \| e^{K^2\lambda\omega}\hat{g}\|_\infty \|e^{-\mu\omega}\|^2_{L^{(2p)'}(\R^{2N})}}{(2\pi)^{2N}\| g\|_2^4}.
	$$
	\end{itemize}
\end{thm}

\begin{proof}
(i) We first observe that, for $f\in\mathcal{S}_\omega(\R^N)$, by H\"older's inequality and \eqref{sub}, for every $\mu\geq 0$ we have
\begin{align}
	|e^{\mu\omega(x)}V_g f(x,\xi)|&\leq\int_{\mathbb{R}^N}e^{\mu\omega(x)}|f(t)||\overline{g(t-x)}| \,dt \notag\\
	&\leq \int_{\mathbb{R}^N} |f(t)||g(t-x)|e^{K\mu(1+\omega(t)+\omega(x-t))}\, dt \notag\\
	& \leq e^{K\mu }\|e^{K\mu \omega} f \|_p \|e^{K\mu\omega}g\|_{p'}. \label{NV-Sp1}
\end{align}
Recall the fundamental identity of the STFT
\begin{equation}\label{FundGabor}
V_gf(x,\xi)=(2\pi)^{-N} V_{\hat{g}}\hat{f}(\xi,-x);
\end{equation}
the same calculations as above then give
\begin{equation}\label{NV-Sp2}
| e^{\mu\omega(\xi)} V_gf(x,\xi)|\leq (2\pi)^{-N} e^{K\mu}\| e^{K\mu\omega}\hat{f}\|_p \| e^{K\mu\omega}\hat{g}\|_{p'}
\end{equation}
for every $\mu\geq 0$. By \eqref{NV-Sp1} and \eqref{NV-Sp2} we get
\begin{align}
e^{\lambda\omega(x,\xi)}|V_gf(x,\xi)|^2 &\leq e^{K\lambda} e^{K\lambda(\omega(x)+\omega(\xi))} |V_gf(x,\xi)|^2 \notag\\
&\leq \frac{e^{K\lambda } e^{2K^2\lambda}}{(2\pi)^N}\|e^{K^2\lambda \omega} f \|_p \|e^{K^2\lambda\omega}g\|_{p'} \| e^{K^2\lambda\omega}\hat{f}\|_p \| e^{K^2\lambda\omega}\hat{g}\|_{p'}. \label{NV-Sp3}
\end{align}
Observe that, by definition of Rihaczek form,
\begin{equation}\label{NV-Sp8}
\|e^{K^2\lambda \omega} f \|_p \| e^{K^2\lambda\omega}\hat{f}\|_p = \| e^{K^2\lambda(\omega(x)+\omega(\xi))}Rf(x,\xi)\|_p \leq \| e^{2K^2\lambda\omega}Rf\|_p,
\end{equation}
where the norms in the left-hand side are in $L^p(\R^N)$ while the ones in the right-hand side are in $L^p(\R^{2N})$; the same holds for $g$, and so by \eqref{NV-Sp3} the proof of point (i) is complete. \\[0.2cm]
\indent (ii) We start by proving that for every $\lambda\geq 0$
\begin{equation}\label{NV-Sp4}
\|e^{\lambda\omega}Rf\|_{\infty}\leq \frac{e^{K(1+2K)\lambda} \| e^{K^2\lambda\omega}g\|_\infty \| e^{K^2\lambda\omega}\hat{g}\|_\infty}{(2\pi)^{2N}\| g\|_2^4} \| e^{K^2\lambda\omega}V_gf\|_1^2.
\end{equation}
We have indeed that, from the inversion formula (\ref{inv}), it follows that for every $y\in\R^N$ and $\lambda\geq 0$,
\begin{align*}
	e^{\lambda\omega(y)}f(y)=\frac{1}{(2\pi)^N \|g\|_2^2}\int_{\R^{2N}} V_g f(x,\xi) e^{\lambda\omega(y)+iy\xi}g(y-x)\, dxd\xi.
\end{align*}
Applying \eqref{sub}, we get for every $y\in\R^N$ and $\lambda\geq0$
\begin{align}
	\| e^{\lambda\omega}f\|_\infty
	&\leq \sup_{y\in\R^N}\left[ \frac{1}{(2\pi)^N \|g\|_2^2}\int_{\R^{2N}} |V_g f(x,\xi)|e^{K\lambda(1+\omega(x)+\omega(y-x))}|g(y-x)|\, dxd\xi\right] \notag\\
	& \leq \frac{e^{K\lambda}\|e^{K\lambda\omega}g\|_\infty}{(2\pi)^N\|g\|_2^2}\int_{\R^{2N}} |V_g f(x,\xi)| e^{K\lambda\omega(x,\xi)}\, dxd\xi\notag \\
	&=\frac{e^{K\lambda}\|e^{K\lambda\omega}g\|_\infty\|e^{K\lambda\omega}V_gf\|_{1}}{(2\pi)^N\|g\|_2^2}. \label{NV-Sp5}
\end{align}
Now we observe that, applying the Fourier transform to both sides of \eqref{inv} we get
\begin{align*}
	\hat{f}(y)=\frac{1}{(2\pi)^N \|g\|_2^2}\int_{\R^{2N}} V_g f(x,\xi)e^{ix\xi}M_{-x}T_{\xi} \hat{g}(y)\, dxd\xi;
\end{align*}
hence proceeding as before we obtain
\begin{equation}\label{NV-Sp6}
	\|e^{\lambda\omega}\hat{f}\|_\infty\leq \frac{e^{K\lambda}\|e^{K\lambda\omega}\hat{g}\|_\infty\|e^{K\lambda\omega}V_gf\|_{1}}{(2\pi)^N\|g\|_2^2}.
\end{equation}
Since
\begin{equation*}
\| e^{\lambda\omega}Rf\|_\infty \leq \| e^{K\lambda(1+\omega(x)+\omega(\xi))}f(x)\hat{f}(\xi)\|_\infty = e^{K\lambda} \| e^{K\lambda\omega}f\|_\infty \| e^{K\lambda\omega}\hat{f}\|_\infty,
\end{equation*}
by \eqref{NV-Sp5} and \eqref{NV-Sp6} with $K\lambda$ instead of $\lambda$ we then obtain \eqref{NV-Sp4}. \\[0.1cm]
Fix now $p\in[1,\infty]$. For every $\lambda\geq 0$ and $\mu>\frac{2N}{b(2p)'}$ we get by H\"older inequality
\begin{align}
	\|e^{\lambda\omega}V_gf\|^2_{1}&=\left(\int_{\R^{2N}}e^{(\lambda+\mu)\omega(x,\xi)-\mu\omega(x,\xi)}|V_gf(x,\xi)|dxd\xi\right)^2\notag \\
	&\leq \|e^{-\mu\omega}\|_{L^{(2p)'}(\R^{2N})}^2\|e^{(\lambda+\mu)\omega}V_gf\|_{2p}^2 \notag \\
	&=\|e^{-\mu\omega}\|_{L^{(2p)'}(\R^{2N})}^2\|e^{2(\lambda+\mu)\omega}Sp_gf\|_{p}. \label{NV-Sp7}
\end{align}
The conclusion then follows from \eqref{NV-Sp4} and \eqref{NV-Sp7}.
\end{proof}

\begin{rem}\label{NV-Rem1}
The estimate \eqref{Blambda} extends \eqref{estrepr}, that is obtained by taking $\lambda=0$ and $p=2$ in \eqref{Blambda}. We observe moreover that the case $\lambda=0$ in Theorem \ref{SptstimaconR} gives estimates where `pure' (unweighted) Lebesgue norms appear everywhere except in the right-hand side of \eqref{Clambda}, where the weight $e^{2\mu\omega}$ with $\mu$ sufficiently large remains. In this sense there is a lack of symmetry in estimates \eqref{Blambda} and \eqref{Clambda}; we shall show in Remark \ref{NV-Rem2} below that this cannot be avoided, in the sense that \eqref{Clambda} with $\mu=0$ is not true for every $\lambda\geq 0$, $1\leq p\leq\infty$, and $f\in\mathcal{S}_\omega(\R^N)$ (even with a constant greater than $D_{\lambda,\mu}^{(2)}$).
\end{rem}

We can now prove an uncertainty principle of Donoho-Stark type related to Theorem \ref{SptstimaconR}.

\begin{prop}[Spectrogram vs. Rihaczek Donoho-Stark principle]\label{NV-SpR}
Let $\omega$ be a weight function, $E\subset\R^{2N}$ be a Lebesgue measurable set, $1\leq p\leq\infty$ and $\varepsilon\in(0,1)$. Suppose that one of the following conditions holds:
\begin{itemize}
\item[(a)]
There exist $f,g\in\mathcal{S}_\omega(\mathbb{R}^N)$, $f,g\neq 0$ and $\lambda\geq 0$ such that
$$
\int_E e^{\lambda\omega(x,\xi)} |Sp_gf(x,\xi)|\,dx d\xi\geq (1-\varepsilon) D_\lambda^{(1)}\|e^{2K^2\lambda\omega}Rf\|_{p}\|e^{2K^2\lambda\omega}Rg\|_{p'},  
$$
where $D_\lambda^{(1)}$ is the constant in \eqref{Blambda}.
\item[(b)]
There exist $f,g\in\mathcal{S}_\omega(\mathbb{R}^N)$, $f,g\neq 0$, $\lambda\geq 0$ and $\mu >\frac{2N}{b(2p)'}$ such that
$$
\int_E e^{\lambda\omega(x,\xi)} |Rf(x,\xi)| \geq (1-\varepsilon) D_{\lambda,\mu}^{(2)}\|e^{2(K^2\lambda+\mu)\omega}Sp_gf\|_{p},
$$
where $D_{\lambda,\mu}^{(2)}$ is the constant in \eqref{Clambda}.
\end{itemize}
Then $m(E)\geq 1-\varepsilon$, where $m(E)$ is the Lebesgue measure of $E$.
\end{prop}

\begin{proof}
(a) From Theorem \ref{SptstimaconR}(i) we have
\begin{align*}
\int_E e^{\lambda\omega(x,\xi)} |Sp_gf(x,\xi)|\,dx\,d\xi&\leq m(E) \| e^{\lambda\omega} Sp_gf\|_\infty \\
&\leq m(E) D_\lambda^{(1)} \| e^{2K^2\lambda\omega} Rf\|_p \| e^{2K^2\lambda\omega} Rg\|_{p'}.
\end{align*}
Since $f,g\neq 0$ we have that both $Rf$, $Rg$, and $Sp_gf$ are different from $0$; then by hypothesis (a) we have $m(E)\geq 1-\varepsilon$. The proof of point (b) is analogous.
\end{proof}
The previous result says, roughly speaking, that in a set $E$ with small measure the spectrogram of a signal $f$ cannot show too large time-frequency contents when compared to the total time-frequency contents measured by the Rihaczek form; similarly, in a small set $E$ the Rihaczek form cannot show too large time-frequency contents when compared to the total time-frequency contents measured by the spectrogram. In the following we see that this is true when spectrogram and Rihaczek are substituted by couples of the most common time-frequency representations, as well as by couples of more general representations in the Cohen class. \\[0.2cm]

Now we consider the $\tau$-Wigner transform (cf. Definition \ref{deftauW}). Observe that, denoting by $\mathfrak{T}_z^{[\tau]}$ the operator acting on a function $F$ on $\R^{2N}$ as
$$
\mathfrak{T}_z^{[\tau]}F(x,t):=F(x+\tau t,x-(1-\tau)t),
$$
for $f\in\mathcal{S}(\mathbb{R}^N)$ and $\tau\in[0,1]$ we can write
\begin{align*}
	\Wig_\tau (f)(x,\xi)=\mathcal{F}_{t\to\xi}\left(\mathfrak{T}_z^{[\tau]} (f\otimes \overline{f})\right).
\end{align*}
Since both the (partial) Fourier transform and $\mathfrak{T}_z^{[\tau]}$ can be extended in a standard way to  (ultra)distributions, we can extend $\Wig_\tau$ to $\mathcal{S}'_\omega(\R^N)$, and from \cite[Theorem 3.4]{MO}, we get that
\begin{align*}
	\Wig_\tau: \mathcal{S}_\omega(\R^N)\to \mathcal{S}_\omega(\R^{2N})\\
	\Wig_\tau: \mathcal{S}'_\omega(\R^N)\to \mathcal{S}'_\omega(\R^{2N}).
\end{align*}
We introduce now the following notations for the dilation of a function. For $G=G(x,\xi)$, $x,\xi\in\R^N$, and $\nu_1,\nu_2\in\R\setminus\{ 0\}$ we write
$$
\left(D_{\nu_1}^{[1]} D_{\nu_2}^{[2]}G\right)(x,\xi):=G(\nu_1 x,\nu_2 \xi)
$$
for a dilation by $\nu_1$ in the first half of variables and a dilation by $\nu_2$ in the second half. Moreover, we write $D_\nu$ for a dilation by the same $\nu$ in all variables, in the sense that for a function $F=F(z)$, $z\in\R^M$, and $\nu\in\R\setminus\{ 0\}$ we write
$$
(D_\nu F)(z):=F(\nu z).
$$
In the discussion about $\tau$-Wigner transform below we need some particular dilations (for $\tau\neq 0,1$), that we write for convenience in the following way:
$$
(A_\tau h)(t):= (D_{\frac{\tau-1}{\tau}}h)(t)=h\left(\frac{\tau-1}{\tau} t\right),
$$
for a function $h$ on $\R^N$, and
$$
V_g^\tau f(x,\xi):=\left( D_{\frac{1}{1-\tau}}^{[1]} D_{\frac{1}{\tau}}^{[2]} V_gf\right)(x,\xi)=V_gf\left( \frac{x}{1-\tau},\frac{\xi}{\tau}\right)
$$
for the STFT $V_gf$, whenever it is defined. With these notations we have that, for $\tau\in(0,1)$, the following identity holds (see \cite[Lemma 6.2]{BODD}):
\begin{equation}\label{wigvgf}
	\Wig_\tau (f)(x,\xi)=\frac{1}{\tau^N}e^{\frac{ix\xi}{\tau}} V_{A_\tau f}^\tau f(x,\xi).
\end{equation}
We can now prove the following result, giving a mutual control between the $\tau$-Wigner transform and the Rihaczek form.

\begin{thm}\label{tauWR}
Let $\omega$ be a weight function, $1\leq p\leq \infty$, and $\tau\in (0,1)$.
\begin{itemize}
\item[(i)]
For every $\lambda\geq 0$ and $f\in\mathcal{S}_\omega (\R^N)$ we have
\begin{equation}\label{wigconR}
\|e^{\lambda\omega}\Wig_\tau (f)\|_\infty^2\leq D_\lambda^{(3)}\|e^{4K^2\lambda\omega}Rf\|_p\|e^{4K^2\lambda\omega}Rf\|_{p'},
\end{equation}
where
$$
D_\lambda^{(3)}=\frac{e^{K(1+2K)\lambda}}{(2\pi)^N (\tau-\tau^2)^N}.
$$
\item[(ii)]
For every $\lambda\geq 0$, $\mu>\frac{2N}{b(2p)'}$ and $f\in\mathcal{S}_\omega(\R^N)$ we have
\begin{equation}\label{Rconwig}
\|e^{\lambda\omega}R f\|_\infty\leq D_{\lambda,\mu}^{(4)}\|e^{2K(K^2\lambda+\mu)\omega}\Wig_\tau (f)\|_p, 
\end{equation}
where
$$
D_{\lambda,\mu}^{(4)}=\inf_{\substack{g\in\mathcal{S}_\omega(\R^N) \\ g\neq 0}} \left[ D_{\lambda,\mu}^{(2)} e^{2K(K^2 \lambda+\mu)} \| e^{2K(K^2\lambda+\mu)\omega} \Wig_{1-\tau}(\tilde{g})\|_1\right];
$$
here $D_{\lambda,\mu}^{(2)}$ is the constant in \eqref{Clambda} and $\tilde{g}(t):=g(-t)$, $t\in\R^N$.
\end{itemize}
\end{thm}

\begin{proof}
(i) We first observe by \eqref{wigvgf} that
\begin{align*}
\|e^{\lambda\omega}\Wig_\tau (f)\|_\infty^2 =\frac{1}{\tau^{2N}}\|e^{\lambda\omega}V_{A_\tau f}^\tau f\|_\infty^2=\frac{1}{\tau^{2N}}\|e^{\lambda D^{[1]}_{(1-\tau)}D^{[2]}_\tau \omega}V_{A_\tau f} f\|_\infty^2.
\end{align*}
Then, proceeding as in \eqref{NV-Sp3} we get
\begin{equation}\label{NV-Wtau1}
\begin{split}
\|e^{\lambda\omega}\Wig_\tau (f)\|_\infty^2&\leq \frac{e^{K(1+2K)\lambda}}{(2\pi)^N \tau^{2N}}\|e^{2K^2\lambda D_{(1-\tau)}\omega}f\|_p\|e^{2K^2\lambda D_{(1-\tau)}\omega}A_\tau f\|_{p'}\times \\ &\quad\times\|e^{2K^2\lambda D_\tau \omega}\hat{f}\|_p\|e^{2K^2\lambda D_{\tau} \omega}\widehat{A_\tau f}\|_{p'}.
\end{split}
\end{equation}
Suppose first that $p>1$. Since
\begin{align*}
\widehat{A_\tau f}= \frac{|\tau|^N}{|\tau-1|^N} D_{\frac{\tau}{\tau-1}} \hat{f},
\end{align*}
we have 
\begin{align}
	\|e^{2K^2\lambda D_{\tau} \omega}\widehat{A_\tau f}\|_{p'}^{p'}&=\frac{|\tau|^{p'N}}{|\tau-1|^{p'N}}\|e^{2K^2\lambda D_{\tau} \omega}D_{\frac{\tau}{\tau-1}} \hat{f}\|_{p'}^{p'} \notag\\
	&=\frac{|\tau|^{(p'-1)N}}{|\tau-1|^{(p'-1)N}} \int_{\R^N}e^{2p'K^2\lambda\omega ((\tau -1) y)}|\hat{f}\left(y\right)|^{p'} dy \notag\\
	&=\frac{|\tau|^{(p'-1)N}}{|\tau-1|^{(p'-1)N}}\|e^{2K^2\lambda D_{\tau-1} \omega}\hat{f}\|_{p'}^{p'}. \label{NV-Wtau2}
\end{align}
Analogously, we get
\begin{equation}\label{NV-Wtau3}
	\|e^{2K^2\lambda D_{(1-\tau)}\omega}A_\tau f\|_{p'}^{p'}= \frac{|\tau|^{N}}{|\tau-1|^{N}}\|e^{2K^2\lambda D_{-\tau}\omega}f\|_{p'}^{p'}.
\end{equation}
From \eqref{NV-Wtau1}, \eqref{NV-Wtau2} and \eqref{NV-Wtau3} we obtain
\begin{equation}\label{NV-Wtau4}
	\begin{split}
	\|e^{\lambda\omega}\Wig_\tau (f)\|_\infty^2&\leq \frac{e^{K(1+2K)\lambda}}{(2\pi)^N(\tau-\tau^2)^N}\|e^{2K^2\lambda D_{(1-\tau)}\omega}f\|_p\|e^{2K^2\lambda D_{-\tau}\omega}f\|_{p'}\times\\
	&\quad \times\|e^{2K^2\lambda D_\tau \omega}\hat{f}\|_p\|e^{2K^2\lambda D_{\tau-1}  \omega}\hat{f}\|_{p'}.
	\end{split}
\end{equation}
If $p=1$, similarly we get \eqref{NV-Wtau4}, since
\begin{align*}
	&\|e^{2K^2\lambda D_\tau \omega}\widehat{A_\tau f}\|_{\infty} = \frac{|\tau|^{N}}{|\tau-1|^{N}}\|e^{2K^2\lambda D_{\tau-1} \omega} \hat{f}\|_{\infty},
	\\&\|e^{2K^2\lambda D_{(1-\tau)}\omega}A_\tau f\|_{\infty} = \|e^{2K^2\lambda D_{-\tau}\omega}f\|_{\infty}.
\end{align*}
Observe now that if $\tau\in(0,1)$, then $D_{(1-\tau)}\omega\leq \omega$ and $D_\tau \omega\leq\omega$, since $\omega$ is increasing; moreover, $D_\tau\omega=D_{-\tau}\omega$ and $D_{\tau-1}\omega=D_{1-\tau}\omega$. Hence, from \eqref{NV-Wtau4} and \eqref{NV-Sp8} we get \eqref{wigconR}.\\[0.2cm]
\indent (ii) By \eqref{Clambda} and \eqref{spwig}, for $f\in\mathcal{S}_\omega(\R^N)$ and $\lambda\geq 0$ we have that, for every $g\in\mathcal{S}_\omega(\R^N)$, $g\neq 0$,
\begin{align*}
	\|e^{\lambda\omega}Rf\|_{\infty}\leq D_{\lambda,\mu}^{(2)}\|e^{2(K^2\lambda+\mu)\omega}Sp_gf\|_{p}=D_{\lambda,\mu}^{(2)}\|e^{2(K^2\lambda+\mu)\omega}(\Wig_{1-\tau}(\tilde{g})\star \Wig_\tau (f))\|_{p},
\end{align*}
with $\mu>\frac{2N}{b(2p)'}$. We observe that by \eqref{sub}
\begin{align}
	&e^{2(K^2\lambda+\mu)\omega(z)}|\Wig_{1-\tau}(\tilde{g})\star \Wig_\tau (f)|(z)\leq \notag\\
	&\leq \int_{\R^{2N}} e^{2(K^2\lambda+\mu)\omega(z)}|\Wig_{1-\tau}(\tilde{g})(y)||\Wig_\tau (f)(z-y)|dy \notag\\
	&\;\leq\int_{\R^{2N}} e^{2K(K^2\lambda+\mu)(1+\omega(y)+\omega(z-y))}|\Wig_{1-\tau}(\tilde{g})(y)||\Wig_\tau (f)(z-y)|dy \notag\\
	&\;= e^{2K(K^2\lambda+\mu)}\left(e^{2K(K^2\lambda+\mu)\omega}|\Wig_{1-\tau}(\tilde{g})|\right)\star \left(e^{2K(K^2\lambda+\mu)\omega}|\Wig_\tau (f)|\right)(z). \label{NV-Wtau5}
\end{align}
Therefore, by Young's inequality we get
\begin{align*}
	\|e^{\lambda\omega}Rf\|_{\infty}&\leq D_{\lambda,\mu}^{(2)}\|e^{2(K^2\lambda+\mu)\omega}(\Wig_{1-\tau}(\tilde{g})\star \Wig_\tau (f))\|_{p}\\
	&\leq D_{\lambda,\mu}^{(2)} e^{2K(K^2\lambda+\mu)}\left\|\left(e^{2K(K^2\lambda+\mu)\omega}|\Wig_{1-\tau}(\tilde{g})|\right)\star \left(e^{2K(K^2\lambda+\mu)\omega}|\Wig_\tau (f)|\right)\right\|_p\\
	&\leq D_{\lambda,\mu}^{(2)} e^{2K(K^2\lambda+\mu)}\left\|e^{2K(K^2\lambda+\mu)\omega}\Wig_{1-\tau}(\tilde{g})\right\|_1\left\|e^{2K(K^2\lambda+\mu)\omega}\Wig_\tau (f)\right\|_p.
\end{align*}
Making the $\inf$ on all $g\in\mathcal{S}_\omega(\R^N)\setminus\{ 0\}$, we then obtain \eqref{Rconwig}.
\end{proof}

Similarly as in Proposition \ref{NV-SpR} we have the following uncertainty principle of Donoho-Stark type related to Theorem \ref{tauWR}. The proof is analogous as the one of Proposition \ref{NV-SpR} and it is omitted.

\begin{prop}[$\tau$-Wigner vs. Rihaczek Donoho-Stark principle]\label{NV-DS2}
	Let $\omega$ be a weight function, $E\subset\R^{2N}$ be a Lebesgue measurable set, $1\leq p\leq\infty$, $\tau\in(0,1)$ and $\varepsilon\in(0,1)$. Suppose that one of the following conditions holds:
	\begin{itemize}
	\item[(a)]
		There exist $f\in\mathcal{S}_\omega(\mathbb{R}^N)$, $f\neq 0$, and $\lambda\geq 0$ such that
		\begin{equation*}
			\int_E e^{\lambda\omega(x,\xi)} |\Wig_\tau (f)(x,\xi)|\,dxd\xi \geq (1-\varepsilon) \sqrt{D_\lambda^{(3)}}\|e^{4K^2\lambda\omega}Rf\|_p^{1/2} \|e^{4K^2\lambda\omega}Rf\|_{p'}^{1/2},
		\end{equation*}
	where $D_\lambda^{(3)}$ is the constant in \eqref{wigconR}.
	\item[(b)]
	There exist $f\in\mathcal{S}_\omega(\mathbb{R}^N)$, $f\neq 0$, $\lambda\geq 0$, and $\mu>\frac{2N}{b(2p)'}$ such that
	$$
	\int_E e^{\lambda\omega(x,\xi)} |R f(x,\xi)|\,dxd\xi\geq (1-\varepsilon) D_{\lambda,\mu}^{(4)}\|e^{2K(K^2\lambda+\mu)\omega}\Wig_\tau (f)\|_p,
	$$
	where $D_{\lambda,\mu}^{(4)}$ is the constant in \eqref{Rconwig}.
	\end{itemize}
	Then $m(E)\geq 1-\varepsilon$, where $m(E)$ is the Lebesgue measure of $E$.
\end{prop}

\begin{rem}\label{NV-Rem2}
As already observed in Remark \ref{NV-Rem1}, in Theorems \ref{SptstimaconR} and \ref{tauWR}, when we estimate the Rihaczek form by the spectrogram and the $\tau$-Wigner, a weight of the kind $e^{\mu\omega}$ appears in the right-hand side of \eqref{Clambda} and \eqref{Rconwig}, with $\mu$ sufficiently large, and this constitutes a lack of symmetry with respect to \eqref{Blambda} and \eqref{wigconR}. We show now that this cannot be avoided in general. Consider for instance \eqref{Rconwig} for $\lambda=0$ and $\tau=1/2$:
$$
\| Rf\|_\infty\leq D_{0,\mu}^{(4)} \| e^{2K\mu\omega}\Wig(f)\|_p
$$
for $\mu>0$ sufficiently large. Consider now $p\geq 2$ and suppose that there exists $D>0$ such that
$$
\| Rf\|_\infty\leq D\| \Wig(f)\|_p
$$
for every $f\in\mathcal{S}_\omega(\R^N)$. Since $\Wig: L^2(\mathbb{R}^N)\times L^2(\mathbb{R}^N)\to L^p(\mathbb{R}^{2N})$ is bounded for every $p\geq 2$, we would obtain
\begin{align*}
	\|f\|_\infty\|\hat{f}\|_\infty=\|Rf\|_\infty \leq C\|f\|_2^2
\end{align*}
for some $C>0$. Then in particular the function $F(f):= \frac{\|f\|_\infty\|\hat{f}\|_\infty}{\|f\|_2^2}$ would be bounded, but this is in contradiction with \cite[Theorem 4.12]{WW}. Observe that also \eqref{Clambda} cannot be satisfied in general for $\mu=0$; indeed, if that were the case, then the proof of Theorem \ref{tauWR}(ii) would give \eqref{Rconwig} with $\mu=0$, that we have already shown that in general is not true.
\end{rem}

From the results that we have proved till now we easily obtain the following mutual control between the spectrogram and the $\tau$-Wigner, together with corresponding uncertainty principle of Donoho-Stark type.
\begin{cor}\label{NV-Cor}
Let $\omega$ be a weight function, $1\leq p\leq\infty$, $\tau\in(0,1)$ and $g\in\mathcal{S}_\omega(\R^N)$, $g\neq 0$.
\begin{itemize}
	\item[(i)] For every $\lambda\geq 0$ and $f\in\mathcal{S}_\omega(\R^N)$ we have
	\begin{align}\label{sp-less-wig}
	\|e^{\lambda\omega}Sp_gf\|_\infty\leq D_{\lambda}^{(5)}\|e^{K\lambda\omega}\Wig_\tau (f)\|_p,
	\end{align}
	where $D_{\lambda}^{(5)}=e^{K\lambda} \| e^{K\lambda\omega}\Wig_{1-\tau}(\tilde{g})\|_{p'}$, with $\tilde{g}(t):=g(-t)$.
	\item[(ii)] For every $\lambda\geq 0$, $\mu>\frac{2N}{b}(K^2 +\frac{2}{(2p)'})$ and $f\in\mathcal{S}_\omega(\R^N)$ we have
	\begin{equation}\label{NV-Wtau6}
	\|e^{\lambda\omega}\Wig_\tau (f)\|_\infty\leq D^{(6)}_{\lambda,\mu} \|e^{(8K^4\lambda+\mu)\omega}Sp_gf\|_p,
	\end{equation}
	where
	\begin{equation*}
	\begin{split}
	D^{(6)}_{\lambda,\mu} &=\frac{e^{\frac{K}{2}(1+2K)\lambda}}{(2\pi)^{\frac{5}{2}N}(\tau-\tau^2)^{\frac{N}{2}}\| g\|_2^4} \inf_{\substack{\mu',\mu'': \\ 2(K^2\mu'+\mu'')=\mu \\ \mu'>N/b,\ \mu''>\frac{2N}{b(2p)'}}} \left[ e^{K(1+2K)(4K^2\lambda+\mu')}\| e^{K^2(4K^2\lambda+\mu')\omega} g\|_\infty \right.\\
	&\qquad\left. \times \| e^{K^2(4K^2\lambda+\mu')\omega}\hat{g}\|_\infty \|e^{-\mu'\omega}\|_2 \| e^{-\mu''\omega}\|_{L^{(2p)'}(\R^{2N})}^2 \right].
	\end{split}
	\end{equation*}
\end{itemize}
\end{cor}

\begin{proof}
(i) By \eqref{spwig}, proceeding as in \eqref{NV-Wtau5} we easily obtain the desired estimate by Young's inequality. \\[0.2cm]
\indent (ii) By \eqref{wigconR} with $p=2$ we obtain
$$
\| e^{\lambda\omega} \Wig_\tau(f)\|_\infty\leq \sqrt{D_\lambda^{(3)}} \| e^{4K^2\lambda\omega} Rf\|_2\leq \sqrt{D_\lambda^{(3)}} \| e^{-\mu'\omega}\|_2 \| e^{(4K^2\lambda+\mu')\omega} Rf\|_\infty
$$
for every $\mu'>\frac{N}{b}$. We then get the conclusion by applying \eqref{Clambda} to $\| e^{(4K^2\lambda+\mu')\omega} Rf\|_\infty$.
\end{proof}
Similarly as in Proposition \ref{NV-SpR}, from Corollary \ref{NV-Cor} we can deduce the following \emph{``spectrogram vs. $\tau$-Wigner Donoho-stark principle''}. The proof is left to the reader.
\begin{cor}\label{2023New}
Let $E\subset\R^{2N}$ be a Lebesgue measurable set, $\varepsilon\in(0,1)$, $1\leq p\leq\infty$, and suppose that one of the following conditions holds:
\begin{itemize}
	\item[(a)]
	There exist $f,g\in\mathcal{S}_\omega(\R^N)$, $f,g\neq 0$, and $\lambda\geq 0$ such that
	$$
	\int_E e^{\lambda\omega(x,\xi)} |Sp_gf(x,\xi)|\,dx d\xi\geq (1-\varepsilon) D_\lambda^{(5)}\|e^{K\lambda\omega}\Wig_\tau (f)\|_p
	$$
	where $D_\lambda^{(5)}$ is the constant in point {\rm (i)} of Corollary \ref{NV-Cor}.
	\item[(b)]
	There exist $f,g\in\mathcal{S}_\omega(\mathbb{R}^N)$, $f,g\neq 0$, $\lambda\geq 0$ and $\mu >\frac{2N}{b}(K^2 +\frac{2}{(2p)'})$ such that
	$$
	\int_E e^{\lambda\omega(x,\xi)} |\Wig_\tau(f)(x,\xi)|\,dxd\xi \geq (1-\varepsilon) D_{\lambda,\mu}^{(6)}\|e^{(8K^4\lambda+\mu)\omega}Sp_gf\|_{p},
	$$
	where $D_{\lambda,\mu}^{(6)}$ is the constant in point {\rm (ii)} of Corollary \ref{NV-Cor}.
\end{itemize}
Then $m(E)\geq 1-\varepsilon$, where $m(E)$ is the Lebesgue measure of $E$.
\end{cor}

\begin{rem}\label{BJ1}
The estimates \eqref{wigconR} and \eqref{NV-Wtau6} can be easily extended, in the case $N=1$, to the Born-Jordan representation \eqref{BJdef}. For instance, from \eqref{BJexpression} and \eqref{wigconR} we have
\begin{equation}\label{tauintegral}
\begin{split}
\| e^{\lambda\omega} \BJ(f)\|_\infty &\leq \int_0^1 \|e^{\lambda\omega} \Wig_\tau(f)\|_\infty\,d\tau \\
&\leq \left[\frac{e^{\frac{K}{2}(1+2K)\lambda}}{\sqrt{2\pi}} \int_0^1\frac{1}{\sqrt{\tau-\tau^2}}\,d\tau\right] \|e^{4K^2\lambda\omega}Rf\|^{1/2}_p\|e^{4K^2\lambda\omega}Rf\|^{1/2}_{p'}.
\end{split}
\end{equation}
Since the $\tau$-integral in the right-hand side of \eqref{tauintegral} is convergent we have
\begin{equation}\label{Dlambda4}
\| e^{\lambda\omega} \BJ(f)\|^2_\infty\leq D_\lambda^{(7)} \|e^{4K^2\lambda\omega}Rf\|_p\|e^{4K^2\lambda\omega}Rf\|_{p'}
\end{equation}
for every $1\leq p\leq \infty$, $\lambda\geq 0$ and $f\in\mathcal{S}_\omega(\R)$. Similarly, from \eqref{NV-Wtau6} we get
$$
\| e^{\lambda\omega}\BJ(f)\|_\infty\leq D^{(8)}_{\lambda,\mu} \| e^{(8K^4\lambda+\mu)\omega} Sp_g f\|_p
$$
for every $1\leq p\leq\infty$, $g\in\mathcal{S}_\omega(\R)\setminus\{ 0\}$, $\lambda\geq 0$, $\mu>\frac{2}{b}(K^2+\frac{2}{(2p)'})$ and $f\in\mathcal{S}_\omega(\R)$.
\end{rem}

The next aim is to prove other results where a representation is controlled by another one; as in the previous cases, each time we have such a control we also have a corresponding uncertainty principle of Donoho-Stark type, similarly as in Propositions \ref{NV-SpR}, \ref{NV-DS2} and Corollary \ref{2023New}. Since there are no substantial differences with respect to those cases, from now on we only give the estimates involving the representations; the statement of the corresponding uncertainty principles is left to the reader. We want to analyze representations in the Cohen class (cf. Definition \ref{NV-DefQ}). First of all, from the previous results it is not difficult to control a representation in the Cohen class by the Wigner transform, the Rihaczek form and the Spectrogram.

\begin{cor}\label{Qsigmastimaconwig}
	Let $\omega$ be a weight function and $1\leq p\leq\infty$.
	\begin{itemize}
		\item[(i)] Fix a kernel $\sigma\in\mathcal{S}'_\omega(\mathbb{R}^{2N})$ satisfying $\|e^{\nu\omega}\sigma\|_{p'}<\infty$ for every $\nu\geq0$. Then for every $\lambda\geq 0$ and $f\in\cS_\omega(\R^N)$ we have
		\begin{align}
		&\|e^{\lambda\omega}Q_\sigma f\|_\infty\leq D_\lambda^{(9)}\|e^{K\lambda\omega}\Wig(f)\|_p, \label{Qsigmaconwig}
		\end{align}
		where $D_\lambda^{(9)}=e^{K\lambda} \| e^{K\lambda\omega}\sigma\|_{p'}$.
		\item[(ii)] Let $\sigma\in\mathcal{S}'_\omega(\mathbb{R}^{2N})$ satisfy $\|e^{\nu\omega}\sigma\|_1<\infty$ for every $\nu\geq0$. Then for every $\lambda\geq 0$, $\mu >\frac{2N}{b(2p)'}$ and $f\in\cS_\omega(\R^N)$ we have
		\begin{eqnarray}
		&\|e^{\lambda\omega}Q_\sigma f\|_\infty^2\leq D_\lambda^{(10)}\|e^{4K^3\lambda\omega}Rf\|_p \|e^{4K^3\lambda\omega}Rf\|_{p'}, \label{Dlambda10}
		\end{eqnarray}
		where $D_\lambda^{(10)}=(\frac{4}{2\pi})^N e^{K(2+K+2K^2)\lambda} \| e^{K\lambda\omega}\sigma\|_1^2$.
		\item[(iii)] Let $\sigma$ be as in point {\rm (ii)} and $g\in\mathcal{S}(\R^N)$, $g\neq 0$. Then for every $\lambda\geq 0$, $\mu>\frac{2N}{b}(K^2 +\frac{2}{(2p)'})$ and $f\in\mathcal{S}_\omega(\R^N)$ we have
		\begin{eqnarray}
		&\|e^{\lambda\omega}Q_\sigma f\|_\infty\leq D^{(11)}_{\lambda,\mu} \|e^{(8K^5\lambda+\mu)\omega}Sp_gf\|_p, \label{neq-Q-Sp}
		\end{eqnarray}
		where $D^{(11)}_{\lambda,\mu}=e^{K\lambda} D_{\lambda,\mu}^{(6)} \| e^{K\lambda\omega}\sigma\|_1$ and $D^{(6)}_{\lambda,\mu}$ is the constant in Corollary \ref{NV-Cor} for $\tau=\frac{1}{2}$.
	\end{itemize}
\end{cor}
\begin{proof}
	(i) Arguing as in the proof of Theorem \ref{tauWR} we have
	\begin{align*}
	\|e^{\lambda\omega}Q_\sigma f\|_\infty&\leq e^{K\lambda}\|(e^{K\lambda\omega}\sigma)\star (e^{K\lambda\omega}\Wig(f))\|_\infty\\&\leq e^{K\lambda}\|e^{K\lambda\omega}\sigma\|_{p'}\|e^{K\lambda\omega}\Wig(f)\|_p.
	\end{align*}
	The points (ii) and (iii) are an easy consequence of point (i) for $p=\infty$ and \eqref{wigconR} and \eqref{NV-Wtau6}, respectively.
\end{proof}
\begin{rem}
	As examples of kernels that satisfy $\|e^{\nu\omega}\sigma\|_{p}<\infty$ for every $\nu\geq0$ and $1\leq p\leq \infty$, we can take $\omega$-ultradifferentiable rapidly decreasing functions of Beurling type.
\end{rem}

Finally, we give estimates between two general time-frequency representations in the Cohen class.
\begin{thm}\label{QsigmastimaconQ}
	Let $\omega$ be a non-quasianalytic weight function and consider $1\leq p\leq\infty$. Fix two kernels $\sigma_1,\sigma_2\in\mathcal{S}_\omega'(\mathbb{R}^{2N})$ such that $\sigma_2\in\cO'_{C,\omega}(\R^{2N})$ with $0\notin\overline{Im(\widehat{\sigma_2})}$, and $\sigma_1\in\cS_\omega(\R^N)$. For every $\lambda\geq0$ there exists a positive constant $D_\lambda^{(12)}$ such that for every $f\in\cS_\omega(\R^N)$
	\begin{align}
	&\|e^{\lambda\omega}Q_{\sigma_1} f\|_\infty\leq D_\lambda^{(12)} \|e^{K\lambda\omega}Q_{\sigma_2}f\|_p. \label{QsigmaconQ}
	\end{align}
\end{thm}
\begin{proof}
	We observe that since $\sigma_2\in\cO'_{C,\omega}(\R^N)$, then $\widehat{\sigma_2}\in \cO_{M,\omega}(\R^N)$ (see \cite[Theorem 6.1]{AC2}). The fact that $\widehat{\sigma_2}$ is a multiplier of $\mathcal{S}_\omega(\mathbb{R}^N)$ together with the assumption $0\notin\overline{Im(\widehat{\sigma_2})}$ implies that $\frac{1}{\widehat{\sigma_2}}\in \cO_{M,\omega}(\R^N)$ (for a proof in the classical case see \cite[Lemma 3.1]{AC3}).
	
	Now, fixed $f\in\cS_{\omega}(\R^N)$ and $\lambda\geq0$, using \eqref{eq.FT-C} we have
	\begin{align*}
	    Q_{\sigma_1}f&=\sigma_1\star \Wig (f)=\mathcal{F}^{-1}(\widehat{\sigma_1}\widehat{\Wig (f)})=\mathcal{F}^{-1}\left(\frac{\widehat{\sigma_1}\widehat{\sigma_2}}{\widehat{\sigma_2}}\widehat{\Wig (f)}\right)\\&=\mathcal{F}^{-1}\left(\frac{\widehat{\sigma_1}}{\widehat{\sigma_2}}\right)\star Q_{\sigma_2}f.
	\end{align*}
	Arguing as in the proof of Theorem \ref{tauWR}, we get
	\begin{align}\label{local-us-est}
	    \|e^{\lambda\omega}Q_{\sigma_1} f\|_\infty\leq e^{K\lambda}\left\|e^{K\lambda\omega} \mathcal{F}^{-1}\left(\frac{\widehat{\sigma_1}}{\widehat{\sigma_2}}\right)\right\|_{p'} \|e^{K\lambda\omega}Q_{\sigma_2}f\|_p,
	\end{align}
	where $ p'$ is the conjugate exponent of $p$. Setting $D_\lambda^{(12)}:=e^{K\lambda}\left\|e^{K\lambda\omega} \mathcal{F}^{-1}\left(\frac{\widehat{\sigma_1}}{\widehat{\sigma_2}}\right)\right\|_{p'}$, we get \eqref{QsigmaconQ}. Observe that $D_\lambda^{(12)}<\infty$ since $\widehat{\sigma_1}\in \cS_\omega(\R^N)$ and $\frac{1}{\widehat{\sigma_2}}\in \cO_{M,\omega}(\R^N)$.
\end{proof}

\begin{ex}
	As examples of kernels $\sigma_2$ satisfying the hypotheses of Theorem \ref{QsigmastimaconQ}, we can consider the following.
	\begin{itemize}
		\item[(1)] Let
		$$
		\sigma_2(t,\eta)=a\delta + e^{-ct^2-d\eta^2},
		$$
		where $c,d>0$, $a\in\mathbb{R}\setminus [-1,0]$, and $\delta$ is the Dirac distribution in $\mathbb{R}^{2N}$. We have
		$$
		\widehat{\sigma_2}(x,\xi)=a+\left(\frac{\pi^2}{cd}\right)^{N/2}e^{-\frac{x^2}{4c}-\frac{\xi^2}{4d}}
		$$
		and the requested conditions are satisfied; then for every $\sigma_1\in\cS_\omega(\R^N)$, we have that \eqref{QsigmaconQ} holds for all $f\in\mathcal{S}_\omega(\mathbb{R}^N)$.
		\item[(2)] Let
		$$
		\sigma_2(x,\xi)=\mathcal{F}^{-1} p(x,\xi),
		$$
		where $\mathcal{F}^{-1}$ is the inverse Fourier transform in $\mathbb{R}^{2N}$ and $p(x,\xi)$ is a non vanishing polynomial, say,
		$$
		p(x,\xi)=\sum_{|\alpha|+|\beta|\leq m} c_{\alpha\beta}x^\alpha\xi^\beta.
		$$
		If $p(x,\xi)\neq 0$ for every $(x,\xi)\in\mathbb{R}^{2N}$, we have, indeed, that the condition $0\notin\overline{Im(\widehat{\sigma_2})}$ is satisfied, and moreover, $\sigma_2\in\cO'_{C,\omega}(\R^{2N})$. Observe that in this case
		\begin{equation*}
		\begin{split}
		Q_{\sigma_2}f(x,\xi) &=\sigma_2\star \Wig (f)(x,\xi) \\
		&=\sum_{|\alpha|+|\beta|\leq m} c_{\alpha\beta}D^\alpha_x D^\beta_\xi (\Wig (f))(x,\xi) = p(D_x,D_\xi)\Wig (f)
		\end{split}
		\end{equation*}
		is the differential operator of symbol $p$ applied to $\Wig (f)$. Then for every $\sigma_1\in\cS_\omega(\R^N)$ we have
		$$
		\| e^{\lambda\omega}Q_{\sigma_1}f\|_\infty\leq D_\lambda^{(12)} \|e^{K\lambda\omega}p(D_x,D_\xi)\Wig (f)\|_p
		$$
		for every $f\in\mathcal{S}_\omega(\mathbb{R}^N)$.
	\end{itemize}
\end{ex}
\begin{rem}
	In the estimates between representations that we have proved, we have always considered in the left-hand sides the sup norm $\|\cdot\|_\infty$, since this is what is needed for the corresponding Donoho-Stark type uncertainty principles. Observe however that we can easily get estimates with $L^q$-norms in the left-hand side by increasing $\lambda$ in the exponential containing the weight $\omega$;  indeed,  given $1\leq q<\infty$ and a function $F$ on $\mathbb{R}^{2N}$, we have for every $\lambda\geq0$
	\begin{align*}
	\|e^{\lambda\omega}F\|_q=\|e^{-\mu\omega}e^{(\lambda+\mu)\omega}F\|_q\leq \|e^{-\mu\omega}\|_q\|e^{(\lambda+\mu)\omega}F\|_\infty,
	\end{align*}
	with $\mu>\frac{2N}{bq}$, using \eqref{eq.Lpspazi}.
\end{rem}

\section{Local uncertainty principles for representations in $\cS_\omega(\R^N)$} \label{Sec-local-up}

As applications of the estimates proved in Section \ref{section-estimates}, we can give uncertainty principles of local type for time-frequency representations in the space $\cS_\omega(\R^N)$. We start by recalling the Price local uncertainty principle for the Fourier transform, cf. \cite{P}.
 \begin{thm}\label{price}
	Let $E\subset \R^N$ be a Lebesgue measurable set and $\alpha>\frac{N}{2}$. Then for every $f\in L^2(\R^N)$, $f\neq 0$, and $\overline{t}, \overline{\xi}\in\R^N$ we have
	 \begin{align}
&	\int_{E} |\hat{f}(\xi)|^2\,d\xi<K'm(E)\|f\|^{2-\frac{N}{\alpha}}_2\||t-\overline{t}|^\alpha f\|_2^{\frac{N}{\alpha}}\label{1stim}\\
&	\int_{E} |f(t)|^2\,dt<K'm(E)\|\hat{f}\|^{2-\frac{N}{\alpha}}_2\||\xi-\overline{\xi}|^\alpha \hat{f}\|_2^{\frac{N}{\alpha}}\label{2stim},
	\end{align}
	where $m(E)$ is the Lebesgue measure of the set $E$ and
	\begin{align}\label{cost}
	K':=\frac{\pi^{\frac{N}{2}}}{\alpha}\left(\Gamma\left(\frac{N}{2}\right)\right)^{-1}\Gamma\left(\frac{N}{2\alpha}\right)\Gamma\left(1-\frac{N}{2\alpha}\right)\left(\frac{2\alpha}{N}-1\right)^{\frac{N}{2\alpha}}\left(1-\frac{N}{2\alpha}\right)^{-1},
	\end{align}
	where $\Gamma$ is the Euler function given by $\Gamma(x):=\int_{0}^{\infty}t^{x-1}e^{-t}\, dt$. The constant $K'$ is optimal, and equality in \eqref{1stim}-\eqref{2stim} is never attained for $f\neq0$.
\end{thm}
The word ``local'' is here referred to the fact that (looking for instance at the inequality \eqref{1stim}) the $L^2$ contents of $\hat{f}$ in a measurable set $E$ must be small if the measure of $E$ is small and/or $f$ is ``concentrated'' (where concentration is measured by the dispersion term $\| |t-\overline{t}|^\alpha f\|_2$). This can be seen as a sort of refinement of the classical Heisenberg uncertainty principle; indeed, the latter says that if $f$ is concentrated, then $\hat{f}$ must be spread out, in the sense that it must have a large variance. This does not exclude that $\hat{f}$ is concentrated in very small sets sufficiently far away from each other; the local uncertainty principle shows that this last possibility is not allowed, by saying that $\hat{f}$ cannot be too concentrated in an arbitrary small set $E$.

The main tool used in \cite{P} to prove Theorem \ref{price} is the next proposition, that we shall use in the following for time-frequency representations.
\begin{prop}\label{sprice}
	For every $\alpha>\frac{N}{2}$, $f\in L^2(\R^N)$ and $\overline{t}\in\R^N$ we have
	\begin{equation*}
	\|f\|_1\leq \sqrt{K'}\|f\|^{1-\frac{N}{2\alpha}}_2\||t-\overline{t}|^\alpha f\|_2^{\frac{N}{2\alpha}},
	\end{equation*}
	where $K'$ is given by \eqref{cost}. In particular, if $f\in L^2(\R^N)$ and $\||t-\overline{t}|^\alpha f\|_2<\infty$ for some $\alpha>\frac{N}{2}$ and $\overline{t}\in\R^N$, then $f\in L^1(\R^N)$.
\end{prop}

In the following we prove uncertainty principles of local type involving different time-frequency representations, where we use weighted norms and a corresponding weighted measure of the set $E$. We start by proving the following result. As in the previous section, $K$ and $b$ always indicate the constants of conditions $(\alpha)$ and $(\gamma)$ of Definition \ref{D.weight}, respectively.
\begin{prop}\label{upfconvg}
	Let $\omega$ be a weight function and fix $g\in \mathcal{S}_{\omega}(\mathbb{R}^N)$, $g\neq0$, $\overline{z}\in\R^{2N}$, and $\alpha>N$. Let moreover $E\subset\R^N$ be a Lebesgue measurable set.
	\begin{itemize}
		\item [(i)] For every $\lambda,\mu\geq0$, $\mu'>\frac{N}{b}$ there exists $C_{\lambda,\mu,\mu'}>0$ such that for every $f\in\cS_\omega(\R^N)$
		\begin{align*}
			\int_{E} e^{2\lambda\omega(\xi)}|\hat{f}(\xi)|^2\, d\xi\leq {C_{\lambda,\mu,\mu'}}D_\mu(E) \|e^{\lambda'\omega}Sp_g f\|_{2}^{1-\frac{N}{\alpha}}\||z-\overline{z}|^\alpha e^{\lambda'\omega}Sp_g f\|_{2}^{\frac{N}{\alpha}},
		\end{align*}
		where
		$$
		\lambda':=2[K(\lambda+\mu)+\mu'],\quad D_\mu(E):=\int_{E}e^{-2\mu\omega(\xi)}\, d\xi.
		$$
		\item [(ii)] For every $\lambda,\mu\geq0$, $\mu'>\frac{N}{b}$ there exists $C_{\lambda,\mu,\mu'}>0$ such that for every $f\in\cS_\omega(\R^N)$
		\begin{align*}
			\int_{E} e^{2\lambda\omega(x)}|f(x)|^2\, dx \leq {C_{\lambda,\mu,\mu'}}D_\mu(E) \|e^{\lambda'\omega}Sp_g f\|_{2}^{1-\frac{N}{\alpha}}\||z-\overline{z}|^\alpha e^{\lambda'\omega}Sp_g f\|_{2}^{\frac{N}{\alpha}},
		\end{align*}
		where $\lambda'$ and $D_\mu(E)$ are as in point (i).
	\end{itemize}
\end{prop}

\begin{proof}
	(i) Fixed $f\in\cS_\omega(\R^N)$, $\lambda,\mu\geq0$ and $\overline{z}\in\R^{2N}$ we have
	\begin{align*}
	    \int_{E} e^{2\lambda\omega(\xi)}|\hat{f}(\xi)|^2\, d\xi=\int_{E} e^{2(\lambda+\mu-\mu)\omega(\xi)}|\hat{f}(\xi)|^2\, d\xi
		\leq D_\mu(E)\left\|e^{(\lambda+\mu)\omega}\hat{f}\right\|^2_\infty,
	\end{align*}
	with $D_\mu(E):=\int_{E}e^{-2\mu\omega(\xi)}\, d\xi$. So, by applying \eqref{NV-Sp6} and \eqref{NV-Sp7} with $p=1$, we have
	\begin{align}\label{local-eq1}
	    \int_{E} e^{2\lambda\omega(\xi)}|\hat{f}(\xi)|^2\, d\xi\leq D_\mu(E) {C'_{\lambda,\mu,\mu'}} \|e^{2[K(\lambda+\mu)+\mu']\omega} Sp_g f\|_1,
	\end{align}
	where
	$$
	C'_{\lambda,\mu,\mu'}:=\left(\frac{e^{K(\lambda+\mu)}\|e^{K(\lambda+\mu)\omega}\hat{g}\|_\infty}{(2\pi)^N \|g\|_2^2}\right)^2 \|e^{-\mu'\omega}\|^2_2,
	$$
	with $\mu' >\frac{N}{b}$. Therefore, using Proposition \ref{sprice} with $\alpha>N$ applied to the function $h:= e^{2[K(\lambda+\mu)+\mu']\omega} Sp_g f$, we obtain that
	\begin{align*}
	    \int_{E} e^{2\lambda\omega(\xi)}|\hat{f}(\xi)|^2\, d\xi \leq {C'_{\lambda,\mu,\mu'}}D_\mu(E) \sqrt{K'} \|e^{\lambda'\omega}Sp_g f\|_{2}^{1-\frac{N}{\alpha}}\||z-\overline{z}|^\alpha e^{\lambda'\omega}Sp_g f\|_{2}^{\frac{N}{\alpha}}&,
	\end{align*}
	with $\lambda':= 2[K(\lambda+\mu)+\mu']$. Setting $C_{\lambda,\mu,\mu'}:=C'_{\lambda,\mu,\mu'} \sqrt{K'} $, we get the thesis.
	
	(ii) Fixed $f\in\cS_\omega(\R^N)$, $\lambda,\mu\geq0$ and $\overline{z}\in\R^{2N}$, applying the same calculation in point (i), we have
\begin{align*}
	    \int_{E} e^{2\lambda\omega(x)} |f(x)|^2 \, dx&= (2\pi)^{-2N} \int_{E} e^{2\lambda\omega(x)} |\hat{\hat{f}}(x)|^2\, dx \\ & \leq (2\pi)^{-2N} {C'_{\lambda,\mu,\mu'}}D_\mu(E) \sqrt{K'} \|e^{\lambda' \omega}Sp_{\hat{g}} \hat{f} \|_{2}^{1-\frac{N}{\alpha}} \||z-\overline{z}|^\alpha e^{\lambda' \omega}Sp_{\hat{g}} \hat{f} \|_{2}^{\frac{N}{\alpha}},
\end{align*}
where $D_\mu(E)$ and $\lambda',\mu'$ are as in point (i).

Now observe that, from \eqref{FundGabor}, $$(2\pi)^{-2N}Sp_{\hat{g}} \hat{f}(x,\xi)=(2\pi)^{-2N}|V_{\hat{g}} \hat{f}(x,\xi)|^2=|V_g f(-\xi,x)|^2=Sp_g f(-\xi,x);$$
so, since $\overline{z}$ is arbitrary, we get the thesis.
\end{proof}

\begin{rem}
Observe that the constant $D_\mu(E)$ in Proposition \ref{upfconvg} does not need to be finite for every measurable set $E$ and constant $\mu\geq0$. Of course, $D_\mu(E)<\infty$ for every $\mu\geq0$ if $m(E)<\infty$, and $D_0(E)$ is the Lebesgue measure of $E$. The quantity $D_\mu(E)$ may be finite also for $E$ with infinite Lebesgue measure; for instance, $D_\mu(E)<\infty$ for every $E$ if $\mu > \frac{N}{2b}$ (see \eqref{eq.Lpspazi}). Similar observations apply for the next results, where other quantities (playing the role $D_\mu(E)$) appear, representing different kind of weighted measures of the set $E$.
\end{rem}

Now we prove local uncertainty principles for the spectrogram $Sp_g f$ of two functions $f,g\in\cS_{\omega}(\R^N)$.

\begin{prop}\label{upsptpricecompleto}
	Let $\omega$ be a weight function and fix $g\in \mathcal{S}_{\omega}(\mathbb{R}^N)$, $g\neq0$, $\overline{z}=(\overline{x},\overline{\xi})\in\R^{2N}$ and $\alpha>\frac{N}{2}$. Let moreover $E\subset\R^{2N}$ be a Lebesgue measurable set. 
	\begin{itemize}
	\item [(i)] For every $\lambda\geq 0$ there exists $C_\lambda>0$ such that for every $\mu\geq 0$ and $f\in\cS_\omega(\R^N)$
	\begin{align*}
&	\int_{E}e^{\lambda\omega(x,\xi)}|Sp_gf(x,\xi)|dxd\xi\leq\\&\;\leq C_\lambda D'_\mu(E)  \|e^{\lambda'\omega}f\|_2^{1-\frac{N}{2\alpha}}\|e^{\lambda'\omega}|x-\overline{x}|^\alpha f\|_2^{\frac{N}{2\alpha}} \|e^{\lambda'\omega}\hat{f}\|_2^{1-\frac{N}{2\alpha}}\|e^{\lambda'\omega}|\xi-\overline{\xi}|^\alpha\hat{f}\|_2^{\frac{N}{2\alpha}},
	\end{align*}
	where
	$$
	\lambda':= K^2(\lambda+\mu),\quad D'_\mu(E):=\int_{E}e^{-\mu\omega(x,\xi)}dx d\xi.
	$$
		\item [(ii)] For every $\lambda\geq0$ there exists $C_\lambda>0$ such that for every $\mu\geq 0$ and $f\in\cS_\omega(\R^N)$
	\begin{align}
		\int_{E}e^{\lambda\omega(x,\xi)}|Sp_g f(x,\xi)| dxd\xi\leq C_\lambda M_\mu(E)  \|e^{\lambda'\omega}f\|_2^{2-\frac{N}{\alpha}}\|e^{\lambda'\omega}|x-\overline{x}|^\alpha f\|_2^{\frac{N}{\alpha}}\label{upvgfconf}
	\end{align}
	where
	$$
	\lambda':= \frac{K}{2}(K\lambda+\mu),\quad M_\mu(E):=\int_{E}e^{K\lambda\omega(\xi)-\mu\omega(x)}dxd\xi.
	$$
\item [(iii)] For every $\lambda\geq0$ there exists $C'_\lambda>0$ such that for every $\mu\geq 0$ and $f\in\cS_\omega(\R^N)$
	\begin{align}
		\int_{E}e^{\lambda\omega(x,\xi)}|Sp_gf(x,\xi)|dxd\xi\leq C'_\lambda M'_\mu(E) \|e^{\lambda'\omega}\hat{f} \|_2^{2-\frac{N}{\alpha}}\|e^{\lambda'\omega}|\xi-\overline{\xi}|^\alpha\hat{f}\|_2^{\frac{N}{\alpha}} \label{upvgfconfcap},
	\end{align}
	where
	$$
	\lambda':=\frac{K}{2}(K\lambda+\mu),\quad M'_\mu(E):=\int_{E}e^{K\lambda\omega(x)-\mu\omega(\xi)}dxd\xi.
	$$
	\end{itemize}
\end{prop}
\begin{proof}
	(i) Fixed $f\in \cS_\omega(\R^N)$, $\overline{z}\in\R^{2N}$ and $\lambda,\mu\geq0$, applying  \eqref{NV-Sp3} with $p=1$ we have
	\begin{align*}
	&\int_{E}e^{\lambda\omega(x,\xi)}|Sp_g f(x,\xi)|dxd\xi\leq D'_\mu(E)  \|e^{(\lambda+\mu)\omega}Sp_g f\|_\infty \\& \; \leq  D'_\mu(E) \frac{e^{K(\lambda+\mu)}e^{2K^2(\lambda+\mu)}}{(2\pi)^N}\|e^{K^2(\lambda+\mu)\omega}f\|_1\|\|e^{K^2(\lambda+\mu)\omega}g\|_{\infty} \|e^{K^2(\lambda+\mu)\omega}\hat{f}\|_1\|e^{K^2(\lambda+\mu)\omega}\hat{g}\|_{\infty},
	\end{align*}
	where $D'_\mu(E):=\int_{E}e^{-\mu\omega(x,\xi)}dx d\xi$. Setting $$C'_\lambda:=\frac{e^{K(\lambda+\mu)}e^{2K^2(\lambda+\mu)}}{(2\pi)^N} \|e^{K^2(\lambda+\mu)\omega}g\|_{\infty} \|e^{K^2(\lambda+\mu)\omega}\hat{g}\|_{\infty}$$ and applying Proposition \ref{sprice} to $\|e^{K^2(\lambda+\mu)\omega}f\|_1$ and $\|e^{K^2(\lambda+\mu)\omega}\hat{f}\|_1$ for $\alpha>\frac{N}{2}$, we get the thesis for $C_\lambda=C'_\lambda K'$.
	
(ii)  Fixed $f\in \cS_\omega(\R^N)$, $\overline{z}\in\R^{2N}$ and $\lambda,\mu\geq0$, applying \eqref{NV-Sp1} with $p=1$ and \eqref{sub} we have
	\begin{align*}
	    &\int_{E}e^{\lambda\omega(x,\xi)}|Sp_g f(x,\xi)|dxd\xi \leq \int_{E}e^{K\lambda(1+\omega(x)+\omega(\xi))-\mu\omega(x)+\mu\omega(x)}|Sp_g f(x,\xi)|dxd\xi\\
	    &\qquad\leq e^{K\lambda}\|e^{(K\lambda+\mu)\omega(x)}Sp_g f\|_\infty \int_{E}e^{K\lambda\omega(\xi)-\mu\omega(x)}dxd\xi\\
	    &\qquad\leq e^{K\lambda+\frac{K}{2}(K\lambda+\mu)} \|e^{\frac{K}{2}(K\lambda+\mu)\omega}f\|_1^2\|e^{\frac{K}{2}(K\lambda+\mu)\omega}g\|_\infty^2\int_{E}e^{K\lambda\omega(\xi)-\mu\omega(x)}dxd\xi.
	\end{align*}
	Setting $M_\mu(E):= \int_{E}e^{K\lambda\omega(\xi)-\mu\omega(x)}dxd\xi$, we get the thesis applying Proposition \ref{price} to $\|e^{\frac{K}{2}(K\lambda+\mu)\omega}f\|_1^2$ for $\alpha>\frac{N}{2}$.
	
(iii) The proof of \eqref{upvgfconfcap} is analogous and so is omitted. 
\end{proof}

\begin{rem}
The quantities $D_\mu(E)$, $D'_\mu(E)$, $M_\mu(E)$, $M'_\mu(E)$ in Propositions \ref{upfconvg} and \ref{upsptpricecompleto} give a measure of the ``size'' of $E$ (observe that, in particular, for $\mu=0$ they all coincide with the Lebesgue measure of $E$ when considering $\lambda=0$ in Proposition \ref{upsptpricecompleto}(ii)-(iii)); moreover, the quantities $\| |z-\overline{z}|^\alpha e^{\lambda'\omega}Sp_gf\|_2$, $\|e^{\lambda'\omega}|x-\overline{x}|^\alpha f\|_2$ and $\|e^{\lambda'\omega}|\xi-\overline{\xi}|^\alpha\hat{f}\|_2$ give (weighted) measures of the dispersions of $Sp_gf$, $f$, and $\hat{f}$, respectively. Then we can interpret the previous results as local uncertainty principles involving time-frequency representations, in the following sense: Proposition \ref{upfconvg} says that in a measurable set $E\subset\mathbb{R}^N$ the contents of $f$ and $\hat{f}$ must be as small as the size of $E$ and/or the (weighted) dispersion of the spectrogram $Sp_gf$ are small. Similarly, Proposition \ref{upsptpricecompleto} interchanges the roles of $Sp_gf$ and $f$, $\hat{f}$, saying that in a measurable set $E\subset\mathbb{R}^{2N}$ the contents of the spectrogram $Sp_gf$ must be as small as the size of $E$ and/or the (weighted) dispersion(s) of $f$ and $\hat{f}$ are small.
\end{rem}

In the following we prove other results of this kind, involving the $\tau$-Wigner transform and representations in the Cohen class. We start by analyzing the $\tau$-Wigner of a function $f \in\cS_{\omega}(\R^N)$.

\begin{prop}\label{wigprice}
	Let $\omega$ be a weight function and fix $g\in \mathcal{S}_{\omega}(\mathbb{R}^N)$, $g\neq0$, $\tau\in(0,1)$ and $\overline{z}=(\overline{x},\overline{\xi})\in\R^{2N}$.
	\begin{itemize}
	\item [(i)]  Let $E\subset\R^{N}$ be a Lebesgue measurable set, and $\alpha>N$. For every $\lambda,\mu \geq0$, $\mu'>\frac{2N}{b}$ there exists $C_{\lambda,\mu,\mu'}>0$ such that for every $f\in\cS_\omega(\R^N)$
		\begin{align*}
	\int_{E} e^{2\lambda\omega(\xi)}|\hat{f}(\xi)|^2\, d\xi\leq {C_{\lambda,\mu,\mu'}}D_\mu(E) \|e^{\lambda'\omega}\Wig_{\tau} (f)\|_{2}^{1-\frac{N}{\alpha}}\||z-\overline{z}|^\alpha e^{\lambda'\omega}\Wig_{\tau} (f)\|_{2}^{\frac{N}{\alpha}},
	\end{align*}
	where
	$$
	\lambda':=2K[K(\lambda+\mu)+\mu'],\quad D_\mu(E):=\int_{E}e^{-2\mu\omega(\xi)}\, d\xi.
	$$
   	\item [(ii)] Let $E\subset\R^{2N}$ be a Lebesgue measurable set, and $\alpha>\frac{N}{2}$. For every $\lambda,\mu \geq0$, $\mu'>\frac{2N}{b}(K^2+2)$ there exists $C_{\lambda,\mu,\mu'}>0$ such that for every $f\in\cS_\omega(\R^N)$
	\begin{align*}
&	\int_{E}e^{\lambda\omega(x,\xi)}|\Wig_{\tau} f(x,\xi)|dxd\xi \\&\leq C_{\lambda,\mu,\mu'}D'_\mu(E)  \|e^{\lambda'\omega}f\|_2^{1-\frac{N}{2\alpha}}\|e^{\lambda'\omega}|x-\overline{x}|^\alpha f\|_2^{\frac{N}{2\alpha}} \|e^{\lambda'\omega}\hat{f}\|_2^{1-\frac{N}{2\alpha}}\|e^{\lambda'\omega}|\xi-\overline{\xi}|^\alpha\hat{f}\|_2^{\frac{N}{2\alpha}},
	\end{align*}
	where
	$$
	\lambda':=K^2[8K^4(\lambda+\mu)+\mu'],\quad D'_\mu(E)=\int_{E}e^{-\mu\omega(x,\xi)}dxd\xi.
	$$
\end{itemize}
\end{prop}
\begin{proof}
	(i) We first observe that, proceeding as in the proof of \eqref{local-eq1}, we get
	\begin{equation}\label{new-last-Sp}
		\int_E e^{2\lambda\omega(\xi)} |\hat{f}(\xi)|^2\,d\xi\leq D_\mu(E)\tilde{C}_{\lambda,\mu,\mu'} \| e^{2[K(\lambda+\mu)+\mu']\omega} Sp_g f\|_\infty
	\end{equation}
	for $\lambda,\mu\geq 0$ and $\mu'>\frac{2N}{b}$. Then, by \eqref{sp-less-wig} we get
	\begin{align*}
	    \int_{E} e^{2\lambda\omega(\xi)}|\hat{f}(\xi)|^2\, d\xi&\leq  {C'_{\lambda,\mu,\mu'}}D_\mu(E)\|e^{\lambda'\omega}\Wig_{\tau} (f)\|_{1},
	\end{align*}	
with $\lambda'=2K[K(\lambda+\mu)+\mu']$. Therefore, using Proposition \ref{sprice} with $\alpha>N$ applied to the function $e^{\lambda'\omega} \Wig_{\tau}(f)$, we obtain the conclusion.
	
	(ii) Fixed $f\in \cS_\omega(\R^N)$, $\overline{z}\in\R^{2N}$ and $\lambda,\mu\geq0$, $\mu'>\frac{2N}{b}(K^2+2)$, applying  \eqref{NV-Wtau6} with $p=\infty$ we have
	\begin{align*}
	&\int_{E}e^{\lambda\omega(x,\xi)}|\Wig_{\tau} f(x,\xi)|dxd\xi\leq D'_\mu(E)  \|e^{(\lambda+\mu)\omega}\Wig_{\tau} (f)\|_\infty \\& \; \leq  D'_\mu(E)D^{(6)}_{\lambda+\mu,\mu'}\| e^{\lambda''\omega}Sp_g f\|_\infty,
	\end{align*}
	where $D'_\mu(E):=\int_{E}e^{-\mu\omega(x,\xi)}dx d\xi$ and $\lambda'':=8K^4(\lambda+\mu)+\mu'$. Proceeding as in Proposition \ref{upsptpricecompleto}(i) we get the thesis.
\end{proof}

Finally, we give similar results for general time-frequency representation in the Cohen class.

\begin{prop}\label{wigprice1}
	Let $\omega$ be a weight function and fix $g\in \mathcal{S}_{\omega}(\mathbb{R}^N)$, $g\neq0$, and $\overline{z}=(\overline{x},\overline{\xi})\in\R^{2N}$.
	\begin{itemize}
	\item [(i)] Let $E\subset\R^{N}$ be a Lebesgue measurable set, and $\alpha>N$. Fix a kernel $\sigma\in\cO'_{C,\omega}(\R^{2N})$ with $0\notin\overline{Im(\widehat{\sigma})}$. Then for every $\lambda,\mu \geq0$, $\mu'>\frac{2N}{b}$ there exists $C_{\lambda,\mu,\mu'}>0$ such that for every $f\in\cS_\omega(\R^N)$
		\begin{align*}
	\int_{E} e^{2\lambda\omega(\xi)}|\hat{f}(\xi)|^2\, d\xi\leq {C_{\lambda,\mu,\mu'}}D_\mu(E)\|e^{\tilde{\lambda}\omega} Q_{\sigma}f\|_{2}^{1-\frac{N}{\alpha}}\||z-\overline{z}|^\alpha e^{\tilde{\lambda}\omega}Q_{\sigma}f\|_2^{\frac{N}{\alpha}},
	\end{align*}
	where
	$$
	\tilde{\lambda}:=2K[K(\lambda+\mu)+\mu'],\quad D_\mu(E):=\int_{E}e^{-2\mu\omega(\xi)}\, d\xi.
	$$
   	\item [(ii)] Let $E\subset\R^{2N}$ be a Lebesgue measurable set, and $\alpha>\frac{N}{2}$. Fix a kernel $\sigma\in\mathcal{S}_\omega'(\mathbb{R}^{2N})$ satisfying $\|e^{\nu\omega}\sigma\|_1 <\infty$ for every $\nu\geq 0$. Then for every $\lambda,\mu \geq0$, $\mu'>\frac{2N}{b}(K^2+2)$ there exists $C_{\lambda,\mu,\mu'}>0$ such that for every $f\in\cS_\omega(\R^N)$
	\begin{align*}
&	\int_{E}e^{\lambda\omega(x,\xi)}|Q_{\sigma} f(x,\xi)|dxd\xi \\&\leq C_{\lambda,\mu,\mu'}D_\mu(E) \|e^{\lambda'\omega}f\|_2^{1-\frac{N}{2\alpha}}\|e^{\lambda'\omega}|x-\overline{x}|^\alpha f\|_2^{\frac{N}{2\alpha}} \|e^{\lambda'\omega}\hat{f}\|_2^{1-\frac{N}{2\alpha}}\|e^{\lambda'\omega}|\xi-\overline{\xi}|^\alpha\hat{f}\|_2^{\frac{N}{2\alpha}},
	\end{align*}
	where
	$$
	\lambda':=K^2[8K^5(\lambda+\mu)+\mu'],\quad D'_\mu(E)=\int_{E}e^{-\mu\omega(x,\xi)}dxd\xi.
	$$
\end{itemize}
\end{prop}
\begin{proof}
	(i) Fixed $f\in\cS_\omega(\R^N)$ and $\lambda,\mu\geq0$, by \eqref{new-last-Sp} we have
	\begin{align*}
	    \int_{E} e^{2\lambda\omega(\xi)}|\hat{f}(\xi)|^2\, d\xi &\leq  {\tilde{C}_{\lambda,\mu,\mu'}}D_\mu(E)\|e^{\lambda'\omega}Sp_g f\|_{\infty}\\
	    &={\tilde{C}_{\lambda,\mu,\mu'}}D_\mu(E)\|e^{\lambda'\omega}\left(\Wig ({\tilde{g}})\star \Wig (f)\right)\|_{\infty},
	\end{align*}
	for $\mu'>\frac{2N}{b}$, where $\lambda'=2[K(\lambda+\mu)+\mu']$ and the last equality follows from \eqref{spwig} with $\tau=\frac{1}{2}$. Set $\sigma_1:=\Wig ({\tilde{g}}) \in \cS_\omega(\R^{2N})$. Since $\sigma_1$ and $\sigma$ satisfy the hypotheses of Theorem \ref{QsigmastimaconQ}, from \eqref{local-us-est} with $p=1$ we then have
	\begin{align*}
	    \int_{E} e^{2\lambda\omega(\xi)}|\hat{f}(\xi)|^2\, d\xi &\leq  {\tilde{C}_{\lambda,\mu,\mu'}}D_\mu(E)\|e^{\lambda'\omega}Q_{\sigma_1} f\|_\infty\\
	    &\leq {\tilde{C}'_{\lambda,\mu,\mu'}}D_\mu(E)\left\|e^{K\lambda'\omega}\mathcal{F}^{-1}\left(\frac{\widehat{\sigma_1}}{\widehat{\sigma}}\right)\right\|_\infty \|e^{K\lambda'\omega}Q_{\sigma}f\|_1.
	\end{align*}
The conclusion then follows from Proposition \ref{sprice} applied to $\|e^{K\lambda'\omega}Q_{\sigma}f\|_1$.
	
	(ii) Fixed $f\in \cS_\omega(\R^N)$ and $\lambda,\mu\geq0$, $\mu'>\frac{2N}{b}(K^2+2)$, applying  \eqref{neq-Q-Sp} with $p=\infty$ we have
	\begin{align*}
	&\int_{E}e^{\lambda\omega(x,\xi)}|Q_\sigma f(x,\xi)|dxd\xi\leq D'_\mu(E)  \|e^{(\lambda+\mu)\omega}Q_\sigma f\|_\infty \\& \; \leq  D'_\mu(E)D^{(11)}_{\lambda+\mu,\mu'}\| e^{\lambda''\omega}Sp_g f\|_\infty,
	\end{align*}
	where $D'_\mu(E):=\int_{E}e^{-\mu\omega(x,\xi)}dx d\xi$ and $\lambda'':=8K^5(\lambda+\mu)+\mu'$. Proceeding as in Proposition \ref{upsptpricecompleto}(i) we get the thesis.
\end{proof}

\begin{rem}
Propositions \ref{wigprice} and \ref{wigprice1} give, for $\Wig_\tau$ and $Q_\sigma$, results that correspond to Propositions \ref{upfconvg}(i) and \ref{upsptpricecompleto}(i) for the spectrogram. Proceeding in a similar way it is not difficult to prove, for $\Wig_\tau$ and $Q_\sigma$, results of the kind of Propositions \ref{upfconvg}(ii) and \ref{upsptpricecompleto}(ii)-(iii).
\end{rem}

\vspace*{5mm}
{\bf Acknowledgments.}
The authors of the present publication are members of the INdAM-GNAMPA group (National Group for Mathematical Analysis, Probability and their Applications). All of them were partially supported by the INdAM-GNAMPA Project 2023 ``Analisi di Fourier e Analisi Tempo-Frequenza di Spazi Funzionali e Operatori", CUP\_E53C22001930001 and by the INdAM-GNAMPA Project 2024 ``Analisi di Gabor ed analisi microlocale: connessioni e applicazioni a equazioni a derivate parziali'', CUP\_E53C23001670001.


\begin{thebibliography}{1}

\bibitem{AC}
A.A.~Albanese, C.~Mele, \emph{Multipliers on $\cS_{\omega}(\R^N)$}, J. Pseudo-Differ. Oper. Appl.  \textbf{12} (2021), Article 35.

\bibitem{AC2}
A.A.~Albanese, C.~Mele, \emph{Convolutors on $\cS_{\omega}(\R^N)$},  Rev. Real Acad. Cienc. Exactas Fis. Nat. Ser. A-Mat.   \textbf{115} (2021), Article 157.

\bibitem{AC3} 
A.A.~Albanese, C.~Mele, \textit{Spectra and ergodic properties of multiplication and convolution operators on the space $\mathcal{S}(\R)$}, Rev. Mat. Complut. \textbf{35} (2022), 739--762. 

\bibitem{B}
G.~Bj\"orck, \emph{Linear partial differential operators and generalized distributions}, Ark. Mat. \textbf{6} (1966), 351--407.

\bibitem{BCO}
P.~Boggiatto, E.~Carypis, A.~Oliaro, \emph{Local uncertainty principles for the Cohen class}, J. Math. Anal. Appl. \textbf{419} (2014), 1004--1022.

\bibitem{BODD}
P.~Boggiatto, G.~De Donno, A.~Oliaro, \emph{Time-frequency representations of Wigner type and pseudo-differential operators}, Trans. Amer. Math. Soc. \textbf{362} (9) (2010), 4955--4981.

\bibitem{BODD2}
P.~Boggiatto, G.~De Donno, A.~Oliaro, \emph{Hudson's theorem for $\tau$-Wigner transforms}, Bull. Lond. Math. Soc. \textbf{45} (6) (2013), 1131--1147.

\bibitem{BJOrpwt}
C.~Boiti, D.~Jornet, A.~Oliaro, \emph{Real Paley-Wiener theorems in spaces of ultradifferentiable functions}, J. Funct. Anal. \textbf{278}, (4) (2020), Article 108348.

\bibitem{BMT}
R.~W. Braun, R.~Meise, and B.~A. Taylor, \emph{Ultradifferentiable functions and {F}ourier analysis}, Results Math. \textbf{17} (3-4) (1990), 206--237.

\bibitem{C}
L.~Cohen, \emph{Time-Frequency Analysis}, Prentice Hall Signal Proc. series, New Jersey, 1995.

\bibitem{CR}
E.~Cordero, L.~Rodino, \emph{Time-frequency analysis of operators}, De Gruyter Studies in Mathematics, 75, De Gruyter, Berlin, 2020.

\bibitem{DG}
M.~de Gosson, \emph{Born-Jordan quantization. Theory and applications}, Fundamental Theories of Physics, 182. Springer, 2016.

\bibitem{DS}
D.L.~Donoho, P.B.~Stark, \emph{Uncertainty principles and signal recovery}, SIAM J. Appl. Math. \textbf{49} (3), (1989), 906--931.

\bibitem{G}
K.~Gr\"ochenig, \emph{Foundations of Time-Frequency analysis}, Birkh\"auser, Boston, 2001.

\bibitem{GZ}
K.~Gr\"ochenig, G.~Zimmermann, \emph{Spaces of test functions via the STFT}, J. Funct. Spaces Appl. \textbf{2} (1) (2004), 25--53.

\bibitem{MO}
C.~Mele, A.~Oliaro, \emph{Regularity of global solutions of partial differential equations in non isotropic ultradifferentiable spaces via time-frequency methods}, J. Differ. Equ. \textbf{286} (2021), 821--855.

\bibitem{P} J.F. Price, \textit{Sharp local uncertainty inequalities}, Studia Math. \textbf{85 }(1987), 37--45.

\bibitem{T}
V.~Turunen, \emph{Born-Jordan time-frequency analysis}, RIMS K\^oky\^{u}roku Bessatsu, B56 (2016), 107--186.

\bibitem{WW}
A.~Widgerson, Y.~Widgerson, \emph{The uncertainty principle: variations on a theme}, Bull. Amer. Math. Soc. (N.S.) \textbf{58} (2) (2021), 225--261.

\end{thebibliography}
\end{document}